\documentclass[12pt]{article}
\usepackage{geometry}                
\usepackage{graphicx,color,mathtools}
\usepackage{amssymb,amsmath,amsthm,mathrsfs}
\usepackage[all,cmtip]{xy}
\usepackage{epstopdf, comment, url}
\usepackage{bm} 
\usepackage{enumitem}

\newtheorem{theorem}{Theorem}[section]
\newtheorem{corollary}[theorem]{Corollary}
\newtheorem{lemma}[theorem]{Lemma}

\newtheorem{innercustomthm}{Theorem}
\newenvironment{customthm}[1]
  {\renewcommand\theinnercustomthm{#1}\innercustomthm}
  {\endinnercustomthm}
  
\theoremstyle{remark}
\newtheorem*{remark}{Remark}

\numberwithin{equation}{section}

\DeclareMathOperator{\crit}{crit}
\DeclareMathOperator{\inn}{Inn}
\DeclareMathOperator{\out}{Out}
\DeclareMathOperator{\sing}{Sing}

\DeclareMathOperator{\Mod}{Mod}

\DeclareMathOperator{\supp}{supp }
\DeclareMathOperator{\diam}{diam}

\DeclareMathOperator{\aut}{Aut}

\DeclareMathOperator{\dist}{dist}

\DeclareMathOperator{\escaping}{escaping}
\DeclareMathOperator{\converging}{converging}

\DeclareMathOperator{\escapingto}{escaping\, to}

\DeclareMathOperator{\thick}{thick}
\DeclareMathOperator{\thin}{thin}

\DeclareMathOperator{\re}{Re}
\DeclareMathOperator{\im}{Im}

\DeclareMathOperator{\IT}{IT}
\DeclareMathOperator{\osc}{osc}
\DeclareMathOperator{\interior}{Int}
\DeclareMathOperator{\hull}{fill}

\title{Critical values of inner functions}
\author{Oleg Ivrii and Uri Kreitner}

\date{October 29, 2023}                                           

\begin{document}

\maketitle

\abstract{Let $\mathscr J$ be the space of inner functions of finite entropy endowed with the topology of stable convergence. We prove that an inner function $F \in \mathscr J$ possesses a radial limit (and in fact, a minimal fine limit) in the unit disk at
$\sigma(F')$ a.e.~point on the unit circle. We use this to show that the singular value measure $\nu(F) = \sum_{c \in \crit F} (1-|c|) \cdot  \delta_{F(c)} + F_*(\sigma(F'))$ varies continuously in $F$. Our analysis involves a surprising connection between Beurling-Carleson sets and angular derivatives.}

\section{Introduction}

Let $\mathbb{D}$ be the unit disk and $\partial \mathbb{D}$ be the unit circle. An inner function is a holomorphic self-map of the unit disk such that for almost every $\theta \in [0,2\pi)$, the radial limit $\lim_{r \to 1} F(re^{i\theta})$ exists and has absolute value 1. In this paper, we study the distribution of the critical values of inner functions.

\subsection{A motivating example}

We open the discussion with a motivating example. Consider the inner function
$$
F(z) = \exp \biggl ( \frac{z-1}{z+1} \biggr ),
$$
which is a universal covering map of the punctured disk. Since $F$ has no critical points in the unit disk in the traditional sense, it also has no critical values. 
Nevertheless, one can assign $F$ ``boundary critical structure'' and a non-trivial ``singular value measure.''

For this purpose, we approximate $F$ 
by a sequence of finite Blaschke products $F_n$ which have a single critical point at $c_n = -1+1/n$ with multiplicity $n$\,:
$$
F_n(z) = \biggl ( \frac{z - c_n}{1-c_n z} \biggr )^{n+1}.
$$
 If we
encode the set of critical points of $F_n$ by the measure
$$
\mu_{F_n} = \sum_{c \in \crit F_n} (1-|c|) \cdot \delta_c = \delta_{c_n},
$$
 then as $n \to \infty$, the critical structures $\mu_{F_n}$ tend weakly to $\mu = \delta_{-1}$.
Similarly, if one encodes the critical values of $F_n$ using the measure
$$
\nu_{F_n} = \sum_{c \in \crit F_n} (1-|c|) \cdot \delta_{F_n(c)} = \delta_0,
$$
then as $n \to \infty$, the critical value measures $\nu_{F_n}$ converge weakly to $\nu = \delta_0$.

\subsection{Inner functions of finite entropy}

In this paper, we assign singular value measures to inner functions of finite entropy, i.e.~to inner functions which satisfy
\begin{equation}
\label{eq:finite-entropy}
 \int_{|z|=1} \log |F'| dm < \infty.
\end{equation}
For the above definition to make sense, we require that $F'$ has a non-tangential limit at almost every point on the unit circle. The  condition (\ref{eq:finite-entropy}) says that $F'$ belongs to the Nevanlinna class $\mathcal N$ and therefore has a $BSO$ decomposition into a Blaschke factor, a singular factor and an outer factor. The astute reader may recall that Nevanlinna functions have $B(S_1/S_2)O$ decompositions; however, the work of Ahern and Clark \cite{ahern-clark} rules out the need for using a singular inner function in the denominator.

In this decomposition, the Blaschke factor $B$ encodes the critical points of $F$, while the singular factor $S$ encodes the boundary critical structure. Putting the $B$ and the $S$ together, we obtain the full critical structure $\inn F' = BS$ of $F$. Alternatively, we can encode the critical structure as a measure on the closed unit disk:
$$
\mu_F = \sum_{c \in \crit F} (1-|c|) \cdot  \delta_{c} + \sigma(F').
$$

\paragraph*{Parametrization of inner functions by critical structures.} The main result from \cite{inner} states that up to post-compositions with M\"obius transformations in $\aut(\mathbb{D})$, inner functions of finite entropy are uniquely determined by their critical structures, and describes which critical structures occur:

\begin{customthm}{A}
\label{main-thm-inner}
The natural map $$F \to \inn(F') \quad : \quad \mathscr J/\aut(\mathbb{D}) \to \inn/ \, \mathbb{S}^1$$
is injective. The image consists of all inner functions of the form $BS_\sigma$ where $B$ is a Blaschke product and $S_\sigma$ is the singular factor associated to a measure $\sigma$ whose support is contained in a countable union of Beurling-Carleson sets.
\end{customthm}

Above, a {\em Beurling-Carleson set} $E \subset \partial \mathbb{D}$ is a closed subset of the unit circle of zero Lebesgue measure whose complement is a union of arcs $\bigcup_k I_k$ with 
\begin{equation}
\label{eq:beurling-carleson}
\sum |I_k| \log \frac{1}{|I_k|}  < \infty.
\end{equation}

\paragraph*{Topology of stable convergence.} The space $\mathscr J$ of inner functions of finite entropy is equipped with the topology of {\em stable convergence} where $F_n \to F$ if $F_n$ converges uniformly on compact subsets to $F$ and the Nevanlinna splitting
$$
\inn F'_n \to \inn F', 
\qquad
\out F'_n \to \out F'
$$
 is preserved in the limit. As explained in \cite[Section 4]{inner}, for general sequences of inner functions, one may lose critical points, i.e.~
\begin{equation}
\label{eq:lose-cps}
\liminf_{n \to \infty} \, \mu_{F_n} \ge \mu_F,
\end{equation}
as well as entropy:
\begin{equation}
\label{eq:lose-entropy}
\liminf_{n \to \infty} \, \int_{\partial \mathbb{D}} \log|F'_n|dm \ge \int_{\partial \mathbb{D}} \log|F'|dm.
\end{equation}
While the last equation resembles Fatou's lemma from real analysis, the proof is subtle since we are dealing with boundary values of analytic functions. The rigorous argument involves Julia's lemma \cite[Corollary 4.10]{mashreghi}. The convergence $F_n \to F$ is stable if and only if $$\lim_{n \to \infty} \, \mu_{F_n} = \mu_F,$$ or equivalently, if
$$
\lim_{n \to \infty} \, \int_{\partial \mathbb{D}} \log|F'_n|dm = \int_{\partial \mathbb{D}} \log|F'|dm.
$$
In this topology, finite Blaschke products are dense in $\mathscr J$.  A particular stable approximating sequence $F_n \to F$ is given in \cite[Theorem 5.4]{craizer}, see also \cite[Lemma 4.3]{inner}. 

\subsection{Singular values}

We may encode the critical values of a finite Blaschke product $F$ using the measure
$$
\nu_F := \sum_{c \in \crit F} (1-|c|) \cdot  \delta_{F(c)}.
$$
For a general inner function $F \in \mathscr J$ of finite entropy, we approximate it by a stable sequence of finite Blaschke products $F_n$ and define
 $\nu_F$ as the limit of $\nu_{F_n}$.
Our main objective is to show that the singular value measure $\nu_F$ is well-defined, i.e.~does not depend on the stable approximating sequence $F_n$\,:

\begin{theorem}
\label{main-thm}
For almost every point $\zeta \in \partial \mathbb{D}$ with respect to the singular measure $\sigma(F')$, the radial limit
$
\lim_{r \to 1} F(r\zeta)
$
exists and lies in the open unit disk. 
Let 
\begin{equation}
\label{eq:nuF-def}
\nu_F \, = \, F_*(\mu_F) \, = \,  \sum_{c \in \crit F} (1-|c|) \cdot  \delta_{F(c)} + F_*(\sigma(F')),
\end{equation}
where we take the pushforward with respect to the radial extension of $F$ to the unit circle.
If $F_n \to F$ is Nevanlinna stable, then the measures $\nu_{F_n}$ converge weakly to $\nu_F$.
\end{theorem}

Theorem \ref{main-thm} is rather surprising since in complex analysis, one often deals with the topology of uniform convergence on compact subsets. When the singular measure $\sigma(F')$ is non-trivial, some of the critical points of $F_n$ must escape to the unit circle. At first glance, it seems rather strange that we are able to say anything about the locations of the critical values even if we know the limiting distribution of the critical points. It turns out that stable convergence is a lot stronger than uniform convergence on compact subsets, see Appendix \ref{sec:compactness}.

We say that a simply-connected domain $\Omega \subset \mathbb{D}$ is {\em thick} at  $\zeta \in \partial \Omega \cap \partial \mathbb{D}$ if the Riemann map $\varphi: \mathbb{D} \to \Omega$ has a non-zero angular derivative at $\zeta$. This condition implies that $\Omega$ contains a truncated Stolz angle of any opening at $\zeta$, but is strictly stronger, see Section \ref{sec:angular-derivatives} for details.
 We say that a function $F$ has a {\em thick limit} $L$ at $\zeta$ if for any $\delta > 0$, the set
$\bigl \{ z \in \mathbb{D}: |F(z) - L| < \delta \bigr \}$
has a connected component $\Omega_\zeta$ that is thick at $\zeta$. Thick limits are also called minimal fine limits, e.g.~see
\cite[Chapter 9.3]{AG01} and \cite[Chapter V.5]{Bas95}.

\begin{theorem}
\label{main-thm2}
For almost every point $\zeta \in \partial \mathbb{D}$ with respect to the singular measure $\sigma(F')$, $F$ has a thick limit at $\zeta$.
\end{theorem}

Intuitively, our proof of continuity of the measure $\nu_F$ in Theorem \ref{main-thm} suggests that thick approach regions act as ``traps'' for critical points.

\subsection{Components}

Central to our argument is the notion of a {\em component} of an inner function, which was used by Pommerenke \cite{pommerenke} to study Green's functions of Fuchsian groups.
Suppose $V$ is a Jordan domain compactly contained in the unit disk and $U$ is a connected component of the pre-image $F^{-1}(V)$. 

\begin{lemma}
\label{component-is-jordan}
Any connected component $U$ of $F^{-1}(V)$ is a Jordan domain.
\end{lemma}

By the above lemma, we can define
$$F_{U} \, = \, F_{U \to V} \, := \, \psi^{-1} \circ F \circ \varphi,$$ where $\varphi, \psi$ are Riemann maps from $\mathbb{D}$ to $U$ and $V$ respectively.

\begin{lemma}
\label{component-is-inner}
The component function $F_{U}$ is an inner function.
  \end{lemma}

The proofs of Lemmas \ref{component-is-jordan} and \ref{component-is-inner} will be given in Appendix \ref{sec:preimage}.

\begin{lemma}
\label{stable-exhaustion}
If  $F \in \mathscr J$ is an inner function of finite entropy, then any component inner function $F_U \in \mathscr J$ also has finite entropy.
 \end{lemma}

\begin{proof}
Since composing a function with M\"obius transformations in $\aut(\mathbb{D})$ does not change whether its derivative is in the Nevanlinna class, we may assume that (i) $F(0) = 0$, (ii) $0 \in U, V$ and (iii) $\varphi(0) = \psi(0) = 0$. With these normalizations, we show that $F_U \in \mathscr J$ with $\mu_{F_U} (\overline{ \mathbb{D}}) \le \mu_F (\overline{ \mathbb{D}})$.

We first consider the case when $F$ is a finite Blaschke product. Since $U$ is a subset of the unit disk, $F_U$ can only have less critical points than $F$, and by the Schwarz lemma, the critical points of $F_U$ are farther from the origin than the corresponding critical points of $F$.

In the general case, we approximate $F$ by a stable sequence of finite Blaschke products $F_n$ with $F_n(0) = 0$. Let $U_n$ be the connected component of $F_n^{-1}(V)$ containing 0. As the domains $U_{n}$ converge in Carath\'eodory sense to $U$, the maps $F_{n, U_n} \to F$ converge uniformly on compact subsets of the disk. Hence,
$$
 \mu_{F_U} (\overline{ \mathbb{D}}) \, \le \, \liminf_{n \to \infty} \mu_{F_{n,U_n}} (\overline{ \mathbb{D}})
 \, \le \, \liminf_{n \to \infty} \mu_{F_{n}} (\overline{ \mathbb{D}})
 \, = \, \mu_F(\overline{ \mathbb{D}})
$$
as desired.
\end{proof}

In light of the above lemma, we may compare the singular measures $\sigma(F')$ and $\sigma(F'_U)$.
Since the measure $\sigma(F'_U)$ is supported on the set $\varphi^{-1}(\partial U \cap \partial \mathbb{D})$, it is more natural to compare $\sigma(F')$ with its pushfoward
$\varphi_*\, \sigma(F'_U)$, which is supported on $\partial U \cap \partial \mathbb{D}$.

\begin{theorem}
\label{radon-nikodym}
Let $(\partial U \cap \partial \mathbb{D})_{\thick}$ denote the set of points on the unit circle where $U$ is thick. Then,
\begin{equation}
\varphi_*\, \sigma(F'_U) = |(\varphi^{-1})'(\zeta)| \, d\sigma(F')|_{(\partial U \cap \partial \mathbb{D})_{\thick}},
\end{equation}
where $|(\varphi^{-1})'(\zeta)|$ is interpreted as the inverse of the angular derivative of $\varphi$ at $\varphi^{-1}(\zeta) \in \partial \mathbb{D}$.
\end{theorem}

\subsection{Exceptional set of an inner function}
\label{sec:exceptional}

It is not uncommon for inner functions to have radial limits in the unit disk.
A point $a \in \mathbb{D}$ is called an {\em exceptional point} for an inner function $F$ if the Frostman shift $F_a(z) = \frac{F(z)-a}{1-\overline{a}F(z)}$ has a non-trivial singular factor $S_{\mu_a} = \sing F_a$. As $\mu_a$ is supported on the set
$$
\mathscr R_a = \Bigl \{ \zeta \in \partial \mathbb{D} : \lim_{r \to 1} F(r\zeta) = a \Bigr \},
$$ 
where $F$ has radial limit $a$, the singular masses of different Frostman shifts $F_a$ are mutually singular.

A classical theorem of Frostman \cite[Theorem 2.5]{mashreghi} says that the exceptional set $\mathscr E$  of an inner function has logarithmic capacity zero. Ahern and Clark observed that in the case of inner functions of finite entropy, the exceptional set is at most countable: as $\sing F_a$ divides $F'_a$, it also divides its inner part  $\inn F_a' = \inn F'$, which shows the inequality
\begin{equation}
\label{eq:svm-frostman}
\sigma(F') \ge \sum_{a \in \mathscr E} \sigma(F_a).
\end{equation}
In particular, $\sigma(F')$ records at least as much information than the collection of Frostman shifts.

 In a private communication, C.~Bishop informed us that there is a non-trivial inner function $F \in \mathscr J$ which has no critical points and possesses an empty exceptional set, i.e.~is indestructible. The construction is similar to the one in \cite{bishop}, but is less intricate. According to \cite{dyakonov-mobius}, $\inn F' = S_\mu$ for some non-zero measure $\mu$ on the unit circle. This example shows that the inequality (\ref{eq:svm-frostman}) could be strict.
 
  In view of  Theorem \ref{main-thm}, $F$ has lots of interesting radial limits in the unit disk, not seen by the exceptional Frostman shifts. Lemma \ref{island-lemma} below suggests that $\sigma(F')$ sees the full collection of radial limits of $F$ in the unit disk.

 \begin{remark}
 Incidentally, Bishop's construction gives an example of an indestructible Blaschke product which is not maximal, answering a question posed in \cite{KR-composition}. 
\end{remark}

\subsection{Island structure}
\label{sec:islands}

Following \cite[Section 25]{heins}, we say that $F$ is of {\em island type} over $V$ if  every connected component of $F^{-1}(V)$ is compactly contained in the unit disk. We say that $F^{-1}(V)$ consists of {\em simple islands} if in addition, $F$ maps each connected component $U \subset F^{-1}(V)$ conformally onto $V$.

\begin{lemma}
\label{island-lemma}
Suppose $F \in \mathscr J$ is an inner function of finite entropy. If $V$ is a Jordan domain compactly contained in the unit disk which has positive distance to the support of $\nu_F$, then $F^{-1}(V)$ consists of simple islands.
\end{lemma}

\begin{proof}
Let $U$ be a connected component of $F^{-1}(V)$. Clearly, $U$ does not contain any critical points of $F$. Theorem \ref{main-thm} tells us that $\sigma(F'_U) = 0$, so that $\inn F'_U$ is trivial. By \cite{dyakonov-mobius}, $F_U$ is a M\"obius transformation, which means that $F$ maps $U$ conformally onto $V$.

Can $\partial U$ touch the unit circle at a point $\zeta \in \partial \mathbb{D}$? Since $U$ is a Jordan domain, by Carath\'eodory's theorem, $F|_U$ extends continuously to $\overline{U}$ and maps $\zeta$ to a point in $\partial V$.
Let $V'$ be a slightly larger Jordan domain, which compactly contains $V$, but still has positive distance to $\supp \nu$, and $U'$ be the connected component of $F^{-1}(V')$ which contains $U$. The above reasoning implies that $F|_{U'}$ extends continuously to $\overline{U'}$ and maps $\zeta$ to a point in $\partial V'$, which is a contradiction.
\end{proof}

\section{Background on angular derivatives}
\label{sec:angular-derivatives}


Suppose $\Omega \subset \mathbb{D}$ is a domain in the unit disk bounded by a Jordan curve. We say that $\Omega$ has an {\em inner tangent} at a point $p \in \partial \Omega \cap \partial \mathbb{D}$ if for any $0 < \theta < \pi$, $\Omega$ contains a truncated Stolz angle of opening $\theta$ with vertex at $p$.

In terms of the conformal map $\varphi: \mathbb{D} \to \Omega$, the domain $\Omega$ possesses an inner tangent at $p$ if and only if 
$$
\lim_{z \to q} \ \arg \frac{\varphi(z) - \varphi(q)}{z-q}
$$
exists, where $q = \varphi^{-1}(p)$. Geometrically, this says that the image of a non-tangential ray ending at $q$ is
 asymptotic to a non-tangential ray ending at $p$ and that angles between different non-tangential rays are preserved. In the literature, this property is known as {\em boundary conformality}\/ or {\em semi-conformality}.

 We say that $\varphi$ has a (non-zero) {\em angular derivative} at $q = \varphi^{-1}(p)$ if
the non-tangential limit
$$
\lim_{z \to q} |\varphi'(z)| = A,
$$
for some real number $A > 0$. One can avoid dealing with the point $q$ by saying that the inverse conformal map $\psi: \Omega \to \mathbb{D}$ has  angular derivative $|\psi'(p)| = A^{-1}$.

While the number $A$ depends on the choice of Riemann map $\varphi$, the existence of the angular derivative does not.
In other words, possessing an angular derivative is an intrinsic property of $(\Omega, p)$, which we record by saying that $\Omega$ is {\em thick} at $p$.

We summarize some basic properties of angular derivatives:

\begin{enumerate}

\item If $\varphi$ has an angular derivative at $p$, then $\Omega$ possesses an inner tangent at $p$.

\item Thickness is a local property: modifying $\Omega$ away from $p$ does not change whether $\Omega$ is thick at $p$.

\item Suppose $\Omega_1 \subset \Omega_2 \subset \mathbb{D}$. If $\Omega_1$ is thick at $p$, then so is $\Omega_2$.
\end{enumerate}

In this section, we describe an if and only if condition for $\Omega$ to be thick at a boundary point
$p \in \partial \Omega \cap \partial \mathbb{D}$ due to Rodin and Warschawski.
As the Rodin-Warshawksi condition involves moduli of curve families, it is not easy to verify directly. To rectify this, we also provide a  geometric characterization for the existence of an angular derivative.

For a discussion of inner tangents and angular derivatives in a more general setting, we refer the reader to \cite[Theorem V.5.7]{garnett-marshall} or \cite{BK}.

\subsection{Angular derivatives in the strip}
\label{sec:angular-derivatives-strip}

It is convenient to analyze the above notions in an infinite strip
$$
\mathcal S = \{-1/2 < \im z < 1/2 \}$$ of width 1.
Let  $\mathcal U \subset \mathcal S$ be a Jordan domain with $0 \in \mathcal U$ and $+\infty \in \partial \mathcal U$ and  $\phi: \mathcal S \to \mathcal U$ be a conformal map which takes $0 \to 0$ and $+\infty \to +\infty$. 

\begin{itemize}
\item 
The {\em inner tangent} condition at $+\infty$ says that for any $\delta > 0$, there exists an $x_{\IT} =  x_{\IT}(\delta) \in \mathbb{R}$ such that
\begin{equation}
\label{eq:modulus-inner-tangency}
[x_{\IT}, \infty) \times (-1/2+\delta,\, 1/2-\delta) \subset \mathcal U.
\end{equation}
We refer to a function $x_{\IT}(\delta)$ for which (\ref{eq:modulus-inner-tangency}) holds as a {\em modulus of inner tangency} for $\mathcal U$. Equivalently,
$$
\im \phi(x+iy) - y \to 0, \qquad \text{as }x \to +\infty.
$$

\item The map $\phi$ has an {\em angular derivative} at $+\infty$ if and only if
$$
\lim_{x \to +\infty} x - \phi(x) = C,
$$
as $x$ tends to $+\infty$ along the real axis. We may think of $C$ as the angular derivative in the strip model. 

By the Schwarz lemma, the hyperbolic distance $d_{\mathcal S}(0, \phi(x)) \le d_{\mathcal S}(0, x)$, which implies that $C \ge 0$.
In fact, the constant $C$ measures how much $\phi(x)$ lags behind $x$\,:
\begin{equation}
\label{eq:angular-derivative-in_S}
C = \frac{1}{\pi} \cdot \lim_{x \to +\infty} \Bigl ( d_{\mathcal S}(0, x) - d_{\mathcal S}(0, \phi(x)) \Bigr ).
\end{equation}
Note that the quantity $d_{\mathcal S}(0, x) - d_{\mathcal S}(0, \phi(x))$ is non-decreasing in $x$. If $\phi$ does not possess an angular derivative, then the limit in (\ref{eq:angular-derivative-in_S}) is infinite.
\end{itemize}

In order to state the Rodin-Warschawksi condition for the existence of angular derivative, we make serveral definitions. We denote the vertical line segments that foliate $\mathcal S$ by $\mathcal S(x) = \mathcal S \cap \{ \re z = x \}$ and write $\mathcal S(x_1,x_2) =  \mathcal S \cap \{ x_1 < \re z < x_2 \}$ for the rectangle  bounded by $\mathcal S(x_1)$ and $\mathcal S(x_2)$.
For $x > 0$, let $\mathcal U(x)$ denote the connected component of $\mathcal U \cap \mathcal S(x)$ which separates $0$ from $+\infty$. Define $\mathcal U(x_1,x_2) \subset \mathcal U$ as the subdomain bounded by $\mathcal U(x_1)$ and $\mathcal U(x_2)$. Note that $\mathcal U(x_1,x_2)$ may stick out of $\mathcal S(x_1,x_2)$. 
We view $\mathcal U(x_1,x_2)$ as a conformal rectangle whose {\em vertical sides} are $\mathcal U(x_1)$ and $\mathcal U(x_2)$, with the remainder of $\partial \mathcal U(x_1,x_2)$ forming the {\em horizontal sides}\/.

\begin{theorem}[Rodin-Warschawksi] 
\label{RW}
The conformal map $\phi: \mathcal S \to \mathcal U$ possesses an angular derivative at $+\infty$ if and only if
\begin{equation}
\label{eq:RW}
\Mod \, \mathcal U(x_1,x_2) - \Mod \, \mathcal S(x_1,x_2) \, = \,  \Mod \, \mathcal U(x_1,x_2) - (x_2 - x_1) \,\to\, 0,
\end{equation}
as $x_1,x_2 \to +\infty$.
\end{theorem}

For a proof, see \cite[Theorem V.5.7]{garnett-marshall}.

\begin{remark}
 When we discuss moduli of conformal rectangles, we refer to the moduli of the vertical curve families which connect the horizontal sides.
 \end{remark}

\subsection{A geometric characterization of thickness}

Let $\mathcal U \subset \mathcal S$ be a simply-connected domain which contains the middle strip
 $\mathbb{R} \times (-1/3,1/3)$ of width $2/3$. To estimate the moduli of the sections $\mathcal U(x_1,x_2)$, we 
 define a family of auxiliary domains $\mathcal U_k$, parameterized by $1/2 \le k \le 2$, although
 we will only use $\mathcal U^- = \mathcal U_{1/2}$ and $\mathcal U^+ =  \mathcal U_2$.
  
For a point $(x,y) = (x, -1/2 + h) \in \mathbb{R} \times (-1/2, -1/3)$, set 
$$
Q(x, y, k) = \biggl [ x - \frac{kh}{2}, \, x + \frac{kh}{2}\biggr ] \times \biggl [-\frac{1}{2}, \, -\frac{1}{2} + kh \biggr ],
$$
while for $(x,y) = (x, 1/2 - h) \in \mathbb{R} \times (1/3, 1/2)$,
we define
$$
Q(x, y, k)  = \biggl [ x - \frac{kh}{2}, \, x + \frac{kh}{2}\biggr ] \times \biggl [\frac{1}{2} - kh, \, \frac{1}{2} \biggr ].
$$
Thus, $Q(x, y, k)$ is a square with side length $kh$ which rests either on the top or the bottom side of $\mathcal S$.
The auxiliary domains are given by
$$
\mathcal U_k := \mathcal S \setminus \bigcup_{x+iy \in \partial \mathcal U} Q(x, y, k) .
$$

\begin{theorem}
\label{RW-geometric-form}
A simply-connected domain $\mathcal U \subset \mathcal S$  which contains the strip
 $\mathbb{R} \times (-1/3,1/3)$ has an angular derivative at $+\infty$ if and only if the Euclidean areas
 \begin{equation}
 \label{eq:h-finite2}
 | \mathcal S(x_1, x_2) \setminus \mathcal U^+(x_1,x_2)| \to 0, \qquad \text{as }x_1,x_2 \to +\infty.
 \end{equation}
  \end{theorem}
  
A similar result was obtained by Rohde and Wong \cite{rohde-wong} for half-plane capacity, where the authors considered a number of other auxiliary regions such as the unit hyperbolic neighbourhood of $\mathcal U$ in $\mathcal S$ and the largest domain contained in $\mathcal U$ bounded by Lipschitz graphs with slope at most 1.

\begin{lemma}
\label{RW-preliminaries}
Let $$
\widetilde{\mathcal U}_k(x_1, x_2) := \mathcal S(x_1, x_2) \setminus \bigcup_{\substack{x+iy \in \partial \mathcal U \\ x_1 < x < x_2}} Q(x, y, k).
$$

{\em (i)} For $1/2 \le k \le 2$, the areas
$| \mathcal S(x_1, x_2) \setminus \widetilde{\mathcal U}_k(x_1,x_2)|$
are comparable.

{\em (ii)} One has
$$
 | \mathcal S(x_1, x_2) \setminus \mathcal U_k(x_1,x_2)| \to 0, \qquad \text{as }x_1,x_2 \to +\infty,
$$
if and only if
$$
 | \mathcal S(x_1, x_2) \setminus \widetilde{\mathcal U}_k(x_1,x_2)| \to 0, \qquad \text{as }x_1,x_2 \to +\infty.
$$
\end{lemma}

\begin{proof} A quick way to see (i) is to use the fact that the Hardy-Littlewood maximal function of an $L^1$ function lies in weak-$L^1$, which tells us that the
area of a union of the squares $Q(x, y, 1/2)$ controls the area of the union of the larger squares $Q(x,y,2)$.

The ``$\Rightarrow$'' direction in (ii) is trivial since $$\widetilde{\mathcal U}_k(x_1,x_2) \supseteq \mathcal U_k(x_1,x_2),$$
 as less squares $Q(x,y,k)$ are removed from $\mathcal S(x_1, x_2)$.
Since all the squares that appear in the construction of the domains $\mathcal U_k(x_1,x_2)$ and 
$\widetilde{\mathcal U}_k(x_1,x_2)$ have sidelengths less than $1/3$,
$$| \mathcal S(x_1, x_2) \setminus \mathcal U_k(x_1,x_2)| \le  | \mathcal S(x_1-1, x_2+1) \setminus \widetilde{\mathcal U}_k(x_1-1,x_2+1)|,$$
which shows the ``$\Leftarrow$'' direction.
 \end{proof}
 
 \begin{remark}
The Euclidean areas
 of $$\mathcal S(x_1, x_2) \setminus \mathcal U_k(x_1,x_2), \qquad k \in [1/2,2],$$ may not be comparable. For instance, if 
 $$\mathcal U = \mathcal S \setminus \bigl ( \{x_2 + \varepsilon \} \times [-1/2, -1/2+2\varepsilon] \bigr ),
 $$
 then $|\mathcal S(x_1,x_2) \setminus \mathcal U^-(x_1,x_2)| = 0$ but $|\mathcal S(x_1,x_2) \setminus \mathcal U^+(x_1,x_2)| > 0$.
\end{remark}

\begin{lemma}
\label{RW-moduli-estimate}
For any $1/2 \le k \le 2$,
$$
 \Mod\, \mathcal U(x_1,x_2) - (x_2-x_1) \asymp | \mathcal S(x_1, x_2) \setminus \widetilde{\mathcal U}_k(x_1,x_2)|.
 $$
\end{lemma}

\begin{proof}
{\em Lower bound.} Since the metric
 $$
\rho_1(z) = \frac{1}{x_2-x_1} \cdot \chi_{\mathcal S(x_1,x_2) \cap  \widetilde{\mathcal U}^-(x_1,x_2)},
 $$
  is admissible for the horizontal path family $\Gamma_{\leftrightarrow} \bigl (\mathcal U(x_1,x_2) \bigr )$, we have
$$\Mod\, \mathcal U(x_1,x_2) \ge 1/A(\rho_1).$$ A short computation shows
\begin{align*}
A(\rho_1) 
& \le \frac{1}{(x_2-x_1)^2} \cdot \mathcal |\mathcal S(x_1,x_2) \cap  \widetilde{\mathcal U}^-(x_1, x_2)|   \\
& = \frac{1}{(x_2-x_1)^2} \cdot \Bigl [ (x_2 - x_1) - |\mathcal S(x_1, x_2) \setminus  \widetilde{\mathcal U}^-(x_1, x_2)|  \Bigr ] \\
& = \frac{1}{(x_2-x_1)} \cdot \biggl [ 1- \frac{1}{x_2-x_1}\cdot  |\mathcal S(x_1, x_2) \setminus  \widetilde{\mathcal U}^-(x_1, x_2)| \biggr ]
\end{align*}
so that
$$
\frac{1}{A(\rho_1)} \lesssim (x_2 - x_1) \cdot \biggl [ 1 + \frac{1}{x_2-x_1}\cdot |\mathcal S(x_1, x_2) \setminus  \widetilde{\mathcal U}^-(x_1, x_2)|\biggr ],
$$
as desired.

\medskip

{\em Upper bound.} To see the upper bound, it is enough to check that the metric
 $$
\rho_2(z) = \chi_{\mathcal S(x_1,x_2)} \cdot
\begin{cases} 
1, & z \in \widetilde{\mathcal U}^+(x_1,x_2), \\
2, &z \in \mathcal U(x_1,x_2) \setminus \widetilde{\mathcal U}^+(x_1,x_2),
\end{cases}
 $$
 is admissible for the vertical path family $\Gamma_{\updownarrow} \bigl (\mathcal U(x_1,x_2) \bigr )$. Suppose a path $\gamma \in \Gamma_{\updownarrow} \bigl (\mathcal U(x_1,x_2) \bigr )$  connects
 $P = (\xi_1, -1/2 + \eta_1)$ and $Q = (\xi_2, 1/2 - \eta_2)$ in $\mathcal U(x_1,x_2)$. As the vertical distance between $P$ and $Q$ is $1 - \eta_1-\eta_2$,
 it is clear that $\ell_{\rho_2}(\gamma) \ge 1-\eta_1-\eta_2$. The deficit is made up by the fact that $\gamma \cap (\mathcal U(x_1,x_2) \setminus \widetilde{\mathcal U}^+(x_1,x_2))$ has length at least $\eta_1 + \eta_2$.
   \end{proof}

Putting Lemmas \ref{RW-preliminaries} and \ref{RW-moduli-estimate} together, we get Theorem \ref{RW-geometric-form}.

\subsection{Some special cases}

For continuous functions $h_1, h_2: \mathbb{R} \to [0,1/6)$, consider the strip domain
$$
\mathcal U \, = \, \mathcal U_{h_1, h_2} \, = \, \bigl \{ (x, y) \, : -1/2 + h_1(x) \, \le \,  y \, \le \, 1/2 - h_2(x) \bigr  \}. 
$$
In this case, the auxiliary domain $\mathcal U^+$ is also a strip domain:
\begin{align*}
\mathcal U^+ \, = \, \mathcal U_{h_1^+,h_2^+} \, = \, \mathcal S
&  \setminus \bigcup_{x_0 \in \mathbb{R}} \bigl [x_0 - h_1(x_0), x_0 + h_1(x_0) \bigr ] \times \bigl [-1/2, -1/2 + 2h_1(x_0) \bigr ] \\
& \setminus \bigcup_{x_0 \in \mathbb{R}} \bigl [x_0 - h_2(x_0), x_0 + h_2(x_0) \bigr ] \times \bigl [1/2 - 2h_2(x_0), 1/2 \bigr ].
\end{align*}

As a consequence of Theorem \ref{RW-geometric-form}, we have:

\begin{corollary}
\label{doubling-strip}
The finiteness of the integral
 \begin{equation}
 \label{eq:h-finite}
  \int_{0}^{\infty} [ h_1(x) + h_2(x) ] dx < \infty
 \end{equation}
  is necessary for $\mathcal U(h_1, h_2)$ to have an angular derivative at $+\infty$. If the functions $h_i$, $i=1,2$ satisfy the doubling condition
 \begin{equation}
 \label{eq:doubling-condition}
 h_i(x) \ge c \cdot h_i(x_0), \qquad |x - x_0| < c \cdot h_i(x_0),
 \end{equation}
 for some $c > 0$, then
it is also sufficient.
\end{corollary}

The following theorem is an analogue of Corollary \ref{doubling-strip} in the unit disk setting:

\begin{theorem}
\label{rodin-warschawski}
Suppose $\Omega = \bigl \{ r\zeta : \zeta \in \partial \mathbb{D}, \, 0 \le r < 1 - h(\zeta) \bigr \}$,
where $h: \partial \mathbb{D} \to [0, 1/2]$ is a continuous function. Assume that $h$ satisfies the doubling condition
 $$
 h(\zeta_1) \ge c \cdot h(\zeta_2), \qquad \text{whenever }|\zeta_2 - \zeta_1| < c \cdot h(\zeta_1),
 $$
 for some $c > 0$. Then, $\Omega$ is thick at $p \in \partial \Omega \cap \partial \mathbb{D}$ if and only if
\begin{equation}
\label{eq:rodin-warschawski}
\int_{\partial \mathbb{D}} \frac{h(\zeta)}{|\zeta-p|^2} |d\zeta| < \infty.
\end{equation}
 \end{theorem}

\section{Stability of angular derivatives}
\label{sec:stability-ad}

The following two lemmas are instances of the following principle: the angular derivative does not change much if we do not significantly alter the global shape of $\Omega$ or the local geometry of $\partial \Omega$ near $\zeta$\,:

\begin{lemma}
\label{slight-modification-of-omega}
Suppose $\Omega \subset \mathbb{D}$ is a Jordan domain and $\zeta \in \partial \Omega \cap \partial \mathbb{D}$. For $\delta > 0$, let 
$\Omega_\delta = \hull \bigl ( \Omega \cup (B(\zeta, \delta) \cap \mathbb{D}) \bigr ) $. Let $\varphi_\delta: \mathbb{D} \to \Omega_\delta$ be a sequence of conformal mappings which converge uniformly on compact sets to a conformal mapping $\varphi: \mathbb{D} \to \Omega$. Then, the angular derivatives $|(\varphi_\delta^{-1})'(\zeta)|$ converge to $|(\varphi^{-1})'(\zeta)|$.
\end{lemma}

In the lemma above, the {\em filling} of a planar set $A \subset \mathbb{C}$ refers to the union of $A$ and the bounded connected components of $\mathbb{C} \setminus A$. Taking the filling ensures that the domain $\Omega_\delta$ is simply-connected, so that the conformal mapping $\varphi_\delta$ is well-defined.

\begin{lemma}
\label{omega-increasing}
Suppose $\Omega_n$ is an increasing sequence of domains in the unit disk, whose union is $\Omega$. Let $\varphi_n: \mathbb{D} \to \Omega_n$ be a sequence of conformal mappings which converge uniformly on compact sets to a conformal mapping $\varphi: \mathbb{D} \to \Omega$. 
If $\Omega_1$ is thick at $\zeta \in \partial \mathbb{D}$, then
$$
|(\varphi_n^{-1})'(\zeta)| \to |(\varphi^{-1})'(\zeta)|.
$$
\end{lemma}

We deduce Lemmas \ref{slight-modification-of-omega} and \ref{omega-increasing} from the following technical lemma which follows from the proof of the Rodin-Warschawski criterion (Theorem \ref{eq:RW}) presented in \cite{garnett-marshall}:
 
 \begin{lemma}
\label{ostrowski}
If $\mathcal U \subset \mathcal S$ possesses an inner tangent at $+\infty$ then the pre-images $\phi^{-1}(\mathcal U(x))$ are asymptotically vertical line segments. More precisely, for any $\varepsilon > 0$, if $x \ge x_{\osc}(\varepsilon)$ is sufficiently large then $\phi^{-1}(\mathcal U(x))$ lies within $\varepsilon$ of  $\mathcal S(\re \phi^{-1}(x))$.
\end{lemma}

\begin{remark}
(i) The function $x_{\osc}(\varepsilon)$ can be chosen uniformly over all domains with the same modulus of inner tangency $x_{\IT}(\delta)$. In particular, $x_{\osc}(\varepsilon)$ can be chosen uniformly over the collection of Jordan domains in $\mathcal S$ which contain
a particular domain $\mathcal U$ that has an inner tangent at $+\infty$.

(ii) Since $\partial \mathcal U$ can be complicated, the images of the vertical segments $\mathcal S(x)$ could be rather wiggly.
\end{remark}

\begin{proof}[Proof of Lemma \ref{slight-modification-of-omega}] Converting to the horizontal strip $\mathcal S$ as in Section \ref{sec:angular-derivatives-strip}, we are given a domain $\mathcal U \subset \mathcal S$ with $0 \in \mathcal U$ and $+\infty \in \partial \mathcal U$, as well as a conformal map $\phi: \mathcal S \to \mathcal U$ which fixes $0$ and $+\infty$. For each $n \ge 1$, we define the domain 
$$
\mathcal U_n = \hull \bigl ( \mathcal U \cup \{ z \in \mathcal S: \re z > n \} \bigr )
$$
and form an analogous conformal map $\phi: \mathcal S \to \mathcal U_n$ which fixes $0$ and $+\infty$.

 By the Schwarz lemma, the angular derivatives $C_n = \lim_{x \to +\infty} x - \phi_n(x)$ form an increasing sequence of positive numbers.
We first consider the case when $\phi$ has an angular derivative at $+\infty$, i.e.~$C = \lim_{x \to +\infty} x - \phi(x) < \infty$.
Given an $\varepsilon > 0$, there exists $x'_0 > 0$ sufficiently large so that
  \begin{equation}
  \label{eq:close-to-1-modulus-estimate}
 (x_2 - x_1) -  \Mod \, \mathcal U(x_1,x_2)  < \varepsilon, \qquad \text{for any } x_1, x_2 \ge x'_0.
  \end{equation}
 As $\mathcal U_n \supset \mathcal U$, the modulus estimate (\ref{eq:close-to-1-modulus-estimate}) a fortiori holds for each $\mathcal U_n$ in place of $\mathcal U$.  
 Now, by Lemma \ref{ostrowski}, we may pick an $x''_0 > 0$ sufficiently large so that for any $x \ge x''_0$,
 \begin{itemize}
\item $\phi^{-1}(\mathcal U(x))$ lies within $\varepsilon$ of $\mathcal S(\phi^{-1}(x))$  and 
 \item $\phi^{-1}_n(\mathcal U(x))$ lies within $\varepsilon$ of $\mathcal S(\phi_n^{-1}(x))$, $n \ge 1$.
 \end{itemize}
 Set $x_0 = \max(x'_0, x''_0)$. Taking $x_1 = x_0$ and $x_2$ large in (\ref{eq:close-to-1-modulus-estimate}) shows that
  \begin{itemize}
\item $\phi^{-1}(\mathcal U(x_0))$ lies within $2\varepsilon$ of $\mathcal S(x_0 - C)$  and 
 \item $\phi^{-1}_n(\mathcal U(x_0))$ lies within $2\varepsilon$ of $\mathcal S(x_0 - C_n)$, $n \ge 1$.
 \end{itemize}
 From the convergence $\phi^{-1}_n(x_0) \to \phi^{-1}(x_0)$, we see that $C_n \to C$ as desired.

We now consider the case when $\phi$ does not possess an angular derivative at infinity. As $\lim_{x \to +\infty} x - \phi(x) = \infty$, for any $C_* > 0$, one can find an $x_* > 0$ so that $\phi(x_*) - x_* > C_*$. Since $\phi_n \to \phi$ converge uniformly on compact subsets of $\mathcal S$, 
$C_n > \phi(x_*) - x_* > C_* - 1$ for all $n$ sufficiently large. In other words, $C_n \to \infty$.
 \end{proof}
 
 The proof of Lemma \ref{omega-increasing} is similar.

\section{Continuity of angular derivatives}

\begin{lemma}
\label{continuity-of-angular-derivatives}
Let $\varphi: \mathbb{D} \to \Omega$ be a conformal map onto a Jordan domain $\Omega \subset \mathbb{D}$. The function $|(\varphi^{-1})'(\zeta)|$ is upper semi-continuous on $\partial \mathbb{D}$,
where we use the convention that $|(\varphi^{-1})'(\zeta)| = 0$ if $\zeta \in (\partial \Omega \cap \partial \mathbb{D})_{\thin}$ or $\zeta \notin \partial \Omega$.
\end{lemma}

The lemma says that for any $M > 0$, the set of points
$\mathcal A_M$
where the angular derivative of $\varphi$ is at most $M$ is closed.

\begin{proof}
For $\zeta \in \partial \mathbb{D}$ and $c > 0$, the set
$$
\mathcal H_c(\zeta) = \biggl \{ z \in \mathbb{D} : \re \frac{\zeta + z}{\zeta - z} > c \biggr \}
$$
is a horoball in the unit disk which rests on $\zeta \in \partial \mathbb{D}$. In this representation, the radius of the horoball decreases as $c$ increases.
According to Julia's lemma \cite[Lemma 4.7]{mashreghi}, a point $\zeta \in \partial \mathbb{D}$ belongs to $\mathcal A_M$ if and only if 
\begin{equation}
\label{eq:julia}
\varphi \bigl (\mathcal H_c(\zeta) \bigr ) \subset \mathcal H_{c/M} \bigl (\varphi(\zeta) \bigr ), \qquad \forall c > 0.
\end{equation}

As $\varphi$ is a conformal map to a Jordan domain, it is continuous up to the boundary. 
Since the mapping property  (\ref{eq:julia}) is preserved under limits, the set $\mathcal A_M$ is closed, which is what we wanted to show.
\end{proof}
 
 The same proof shows:

\begin{lemma}
\label{continuity-of-angular-derivatives2}
Let
 $\varphi_n: \mathbb{D} \to \Omega_n$ be a sequence of conformal maps which converge to a conformal map $\varphi: \mathbb{D} \to \Omega$. 
 If $\zeta_n \in \partial \mathbb{D}$ are points on the unit circle which converge to $\zeta \in \partial \mathbb{D}$, then 
 $\limsup_{n \to \infty} |(\varphi_n^{-1})'(\zeta_n)| \le |(\varphi^{-1})'(\zeta)|.$
 \end{lemma}

One can also prove Lemmas \ref{continuity-of-angular-derivatives} and \ref{continuity-of-angular-derivatives2} using Aleksandrov-Clark measures. For the definition and basic properties of Aleksandrov-Clark measures, we refer the reader to the surveys \cite{PS, saksman}.

\section{Green's functions}

Let $\Omega \subset \partial \mathbb{D}$ be a simply-connected domain in the unit disk.
Fix a basepoint $p \in \Omega$ and suppose $\varphi: \mathbb{D} \to \Omega$ is a conformal map which fixes $p$.
In this section, we give an interpretation of the angular derivative in terms of the ratios of the Green's functions  $G_\Omega(p, z)$ and $G_\mathbb{D}(p, z)$ of $\Omega$ and $\mathbb{D}$ respectively:

\begin{lemma}
\label{green-quotient}
Suppose  $c_n \in \Omega$ is a sequence of points that escape to $\zeta \in \partial \Omega \cap \partial \mathbb{D}$.

{\em (a)}
If $\zeta \in \partial \Omega \cap \partial \mathbb{D}$ is a point of thickness of $\Omega$ then
$$
\limsup_{n \to \infty} \frac{ G_{\Omega}(p, c_n) }{ G_{\mathbb{D}}(p, c_n) } \le |(\varphi^{-1})'(\zeta)|.
$$
Equality holds if each $c_n \to \zeta$ non-tangentially.

{\em (b)}
If $\zeta \in \partial \Omega \cap \partial \mathbb{D}$ is not a point of thickness of $\Omega$ then
$$
\limsup_{n \to \infty} \frac{ G_{\Omega}(p, c_n) }{ G_{\mathbb{D}}(p, c_n) } = 0.
$$
\end{lemma}

\begin{proof}
  Let $\zeta \in \partial \Omega \cap \partial \mathbb{D}$ be a point of thickness of $\partial \Omega$. Fix $\delta > 0$. Since $c_n \to \zeta$, for $n$ large, the points $c_n$ are contained in $B(\zeta, \delta)$. Since by increasing the domain, one increases the Green's function, we have
\begin{equation}
\label{eq:tilde-comparison}
 G_{\Omega}(p, c_n) \le  G_{\Omega_\delta}(p, c_n),
 \end{equation}
 where $\Omega_\delta = \hull \bigl ( \Omega \cup (B(\zeta, \delta) \cap \mathbb{D}) \bigr )$ as in Section \ref{sec:stability-ad}. 
 
  Let $\varphi_\delta$ denote the Riemann map from $\mathbb{D}$ to $\Omega_\delta$ which fixes the point $p$.
As the Green's function is conformally invariant, for any fixed $\delta > 0$,
$$
G_{\Omega_\delta}(p, c_n) \, = \,  G_{\mathbb{D}} \bigl (p, {\varphi}^{-1}_{\delta}  (c_n) \bigr )
\, = \, G_{\mathbb{D}}(p, c_n) \cdot \bigl (|({\varphi}^{-1}_{\delta} )'(\zeta)| + o (1) \bigr ),
$$
as $n \to \infty$. In view of Lemma \ref{continuity-of-angular-derivatives2}, when $n > n_0(\zeta, \delta)$ is large, we have
 $$
 G_{\Omega}(p, c_n) \le    G_{\mathbb{D}}(p, c_n) \cdot \bigl ( |(\varphi^{-1})'(\zeta)| + o_\delta(1) \bigr ),
 $$
 where the term $o_\delta(1)$ can be made arbitrarily small by choosing $\delta > 0$ small. Part (b) is similar.
 \end{proof}

An analogue of Lemma \ref{green-quotient} holds for sequences:

\begin{lemma}
\label{green-quotient2}
Suppose $(\Omega_n, p) \to (\Omega, p)$ is a sequence of domains in the unit disk converging in the Carath\'eodory sense and  $c_n \in \Omega_{n}$ is a sequence of points that escape to $\zeta \in \partial \mathbb{D}$.

{\em (a)} If $\zeta \in \partial \Omega \cap \partial \mathbb{D}$ is a point of thickness of $\Omega$ then
$$
\limsup_{n \to \infty} \frac{ G_{\Omega_{n}}(p, c_n) }{ G_{\mathbb{D}}(p, c_n) } \le |(\varphi^{-1})'(\zeta)|.
$$

{\em (b)} If $\zeta \in \partial \Omega \cap \partial \mathbb{D}$ is not a point of thickness of $\Omega$, 
then
$$
\limsup_{n \to \infty} \frac{ G_{\Omega_{n}}(p, c_n) }{ G_{\mathbb{D}}(p, c_n) } = 0.
$$
\end{lemma}

The proof is essentially the same. 
While we are on the subject of Green's functions, we mention the following well-known lemma:

\begin{lemma}
\label{green-quotient3}
Let $p, q \in \mathbb{D}$ and $\zeta \in \partial \mathbb{D}$. For any sequence of points $c_n \in \mathbb{D}$ converging to $\zeta$,
$$
\lim_{n \to \infty} \frac{G_{\mathbb{D}}(p, c_n)}{G_{\mathbb{D}}(q, c_n)} = \frac{P_p(\zeta)}{P_q(\zeta)}.
$$
\end{lemma}

\section{Composition operator on measures}

Suppose $\Omega \subset \mathbb{D}$ is a Jordan domain. By Carath\'eodory's theorem, any conformal map $\varphi: \mathbb{D} \to \Omega$ extends continuously to the closed unit disk $\overline{\mathbb{D}}$.  
 For a positive measure $\mu \ge 0$ on the unit circle, we can form its Poisson extension $P_\mu(z)$ to the unit disk.  Since $P_{\mu}(\varphi(z))$ is a positive harmonic function, it can be represented as the Poisson extension of some finite measure $\nu \ge 0$.

By decomposing  $\mu = \mu_+ - \mu_-$ into positive and negative parts, the mapping $\mu \to \nu$ naturally extends to signed measures. We refer to the correspondence $\mu \to \nu$ as the {\em composition operator on measures}\/ 
$C_\varphi: \mathcal M(\partial \mathbb{D}) \to  \mathcal M(\partial \mathbb{D})$. When $\mu = \delta_x$ is a delta mass, $\nu_x := C_\varphi \delta_x$ is known as the {\em Aleksandrov-Clark measure} at $x$. It is easy to see that $C_\varphi$ is a continuous linear operator when $\mathcal M(\partial \mathbb{D})$ is equipped with the weak topology on measures.

In this section, we give an explicit formula for $C_\varphi \mu$. Below, we write $m = |dz|/2\pi$ for the normalized Lebesgue measure on the unit circle and $d\omega_{\Omega, w}$ for the harmonic measure on $\partial \Omega$ as viewed from $w \in \Omega$. The following theorem is taken from \cite[Theorem 2.1]{BBC84}:

\begin{theorem}
\label{radon-nikodym2}
If $\mu$ is a finite meaure on the unit circle and $\nu = C_\varphi \mu$, then 
\begin{equation*}
\varphi_*\nu|_{(\partial \Omega \cap \mathbb{D})} = P_\mu(z) \, d\omega_{\Omega, \varphi(0)}(z), \qquad
\varphi_* \nu|_{(\partial \Omega \cap \partial \mathbb{D})} = 
|(\varphi^{-1})'(x)| \, d\mu(x).
\end{equation*}
\end{theorem}

\begin{proof} 
{\em Step 1.} In this step, we prove the theorem when $\mu =\delta_x$ is a delta mass.  To this end, we consider two cases:

Case I. If $x \notin \partial \Omega \cap \partial \mathbb{D}$, then $P_{\delta_x}(\varphi(z))$ extends continuously to the unit circle and  $\nu_x = P_{\delta_x}(\varphi(z))dm$.

Case II. Otherwise, $x = \varphi(y)$ for some point $y \in \partial \mathbb{D}$. 
Since $P_{\delta_x}(\varphi(z))$ extends continuously to $\partial \mathbb{D} \setminus \{y\}$, we have
 $$\nu_x|_{\varphi^{-1}(\partial \Omega \cap \mathbb{D})} = P_{\delta_x}(\varphi(z))dm, \qquad \supp \nu_x|_{\varphi^{-1}(\partial \Omega \cap \partial \mathbb{D})} \subseteq \{y\}.$$
To evaluate $\nu_x(\{y\})$, we consider two sub-cases:

\begin{enumerate}
\item[a.] Suppose first that $x \in (\partial \Omega \cap \partial \mathbb{D})_{\thick}$. Since $\varphi$ has an angular derivative at $y = \varphi^{-1}(x)$,
$$
P_{\delta_{x}}(\varphi(ry)) \sim  |(\varphi^{-1})'(x)| \cdot \frac{2}{1-r}, \qquad \text{as }r \to 1,
$$
which tells us that $\nu_x(\{y\}) = |(\varphi^{-1})'(x)|$.
\item[b.]
On the other hand, when $x \in (\partial \Omega \cap \partial \mathbb{D})_{\thin}$,
$$
P_{\delta_{x}}(\varphi(ry)) = o \biggl ( \frac{1}{1-r} \biggr ), \qquad \text{as }r \to 1,
$$
which implies that  $\nu_x(\{y\}) = 0$.
\end{enumerate}
In all cases, $\nu_x$ matches with the expression in the statement of theorem. 

\medskip

{\em Step 2.}
To deduce the general case from that of a delta mass, we expand the expression for
$$
P_{\mu}(\varphi(z)) \, = \, \int_{\partial \mathbb{D}} P_{\nu_x}(z) d\mu(x),
$$
which is naturally a sum of two terms: $\textrm{I} +  \textrm{II}$, corresponding to the cases $\textrm{I}$ and $\textrm{II}$ above.
To simplify the expression for the first term, we use Fubini's theorem:
\begin{align*}
 \textrm{I}  & \, = \, \frac{1}{2\pi} \int_{\partial \mathbb{D}} \biggl \{ \int_{\partial \mathbb{D}} \re \frac{\zeta+z}{\zeta-z} \cdot \bigl [ P_{\delta_x}(\varphi(\zeta)) dm(\zeta) \bigr ] \biggr \} d\mu(x) \\
& \, = \,  \frac{1}{2\pi}  \int_{\partial \mathbb{D}} \re \frac{\zeta+z}{\zeta-z} \cdot \bigl [ P_{\mu}(\varphi(\zeta)) dm(\zeta) \bigr ],
\end{align*}
while for the second term, we make a change of variables: 
\begin{align*}
 \textrm{II} & \, = \, \frac{1}{2\pi} \int_{(\partial \Omega \cap \partial \mathbb{D})_{\thick}} \biggl \{ \int_{\partial \mathbb{D}}  \re \frac{\zeta+z}{\zeta-z} \cdot  \, \bigl [ |(\varphi^{-1})'(x)| \, \delta_y (\zeta) \bigr ] \biggr \} d\mu(x) \\
 & \, = \, \frac{1}{2\pi} \int_{(\partial \Omega \cap \partial \mathbb{D})_{\thick}}  \re \frac{y+z}{y-z} \cdot  \, |(\varphi^{-1})'(x)| \, d\mu(x), \\
& \, = \, \frac{1}{2\pi} \int_{\varphi^{-1}((\partial \Omega \cap \partial \mathbb{D})_{\thick})}  \re \frac{y+z}{y-z} \cdot  d\varphi^* \Bigl [ |(\varphi^{-1})'(x)| \, d\mu(x) \Bigr ] (y).
\end{align*}
As the expressions for $ \textrm{I}$ and $ \textrm{II}$ above match the expressions in the statement of the theorem, the  proof is complete.
\end{proof}

\section{Radon-Nikodym derivative}

Let $F \in \mathscr J$ be an inner function of finite entropy.
Suppose $V$ is a Jordan domain compactly contained in the unit disk. Let $U$ be a connected component of $F^{-1}(V)$ and $F_U = \psi^{-1} \circ F \circ \varphi$ be the associated component inner function.
In this section, we prove Theorem \ref{radon-nikodym} which relates the singular parts of $F'$ and $F'_U$.

We say that a Jordain $\Omega \subset \mathbb{C}$ is a {\em Nevanlinna domain} if the derivative of any conformal map $R : \mathbb{D} \to \Omega$ is an outer function. In the literature, there is a similar notion of a Smirnov domain which also requires $\partial \Omega$ to be a rectifiable. It is easy to see that any smooth domain is a Nevanlinna domain. In fact, any chord-arc curve is a Smirnov domain and hence a Nevanlinna domain, e.g.~see \cite[Theorem VII.4.6]{garnett-marshall}.
As a byproduct of our investigation, we will obtain the following theorem:

\begin{theorem}
\label{nevanlinna-preserving}
Suppose $F \in \mathscr J$ is an inner function of finite entropy. If $V$ is a Nevanlinna domain, then so is any connected component 
$U$ of $F^{-1}(V)$.
 \end{theorem}

\subsection{Upper bound}
\label{sec:upper-bound}
Let $F_n \to F$ by a stable approximation by finite Blaschke products. For each $n = 1, 2, \dots$, we  select a connected component ${U_n} \subset F_n^{-1}(V)$ so that the $U_n$ converge to $U$ in the Carath\'eodory topology.

 Choose an arbitrary basepoint $p \in U$. By dropping finitely terms from the sequence, we may assume that $p \in U_n$ for all $n$.
For each $n = 1, 2, \dots$, we form a conformal map $\varphi_n: \mathbb{D} \to U_n$ with $\varphi_n(p) = p$ and $\varphi'_n(p) > 0$. By construction, the limit $\varphi: \mathbb{D} \to U$ will be a conformal map which also fixes $p$.  With the above normalization, the associated component inner functions $F_{n, U_n}$ converge uniformly on compact subsets to $F_U$.
  
By Hurwitz theorem, for any critical point $c_i \in \crit F$, one can find a sequence of critical points $c_{n,i} \in \crit F_n$ such that $c_{n,i} \to c_i$.  
For each $n = 1, 2, \dots$, we designate a subset of the critical points of $F_n$ as {\em converging} so that
$$
\sum_{\substack{ c \in \crit F_n \\ \converging}} (1- |c|) \cdot \delta_c \to \sum_{c \in \crit F} (1- |c|) \cdot \delta_c.
$$
We refer to the remaining critical points of $F_n$ as {\em escaping}\/. 

\begin{proof}[Proof of Theorem \ref{radon-nikodym}: Upper bound]
Since $F_n \to F$ is stable, the measures
 $$
\sum_{\substack{ c \in \crit F \\ \escaping}} G_{\mathbb{D}}( 0,  c) \cdot \delta_{c} \, \to \, \sigma(F').
 $$
Applying Lemmas \ref{green-quotient3} and \ref{green-quotient2} respectively, we get
  $$
\sum_{\substack{ c \in \crit F \\ \escaping}} G_{\mathbb{D}}(p,  c) \cdot \delta_{c} \, \to \, P_{p}(\zeta) \, d\sigma(F')
 $$
and
\begin{equation*}
\limsup_{n \to \infty} \sum_{\substack{ c \in \crit F_n \cap U_{n} \\ \escaping}} G_{U_{n}}( p,  c) \cdot \delta_{c}
 \, \le \, |(\varphi^{-1})'(\zeta)| \,  P_{p}(\zeta) \, d\sigma(F')|_{(\partial U \cap \partial \mathbb{D})_{\thick}}.
 \end{equation*} 
As $F_{n,U_n} \to F_U$ converge uniformly on compact subsets of the unit disk, we have
\begin{align*}
\sigma(F'_U)
& \le \liminf_{n \to \infty} \sum_{\substack{\hat c \in \crit F_{n, U_n} \\ \escaping}} G_{\mathbb{D}}( 0, \hat c) \cdot \delta_{\hat c} \\
& = \liminf_{n \to \infty} \sum_{\substack{\hat c \in \crit F_{n, U_n} \\ \escaping}} G_{\mathbb{D}}( p, \hat c) \cdot P_p(\zeta)^{-1} \cdot \delta_{\hat c}.
 \end{align*}
Since one has a  correspondence between critical points $c \in \crit F_n \cap U_{n}$ and $\hat{c} = \varphi_n^{-1}(c) \in \crit F_{n, U_n}$, the conformal invariance of the Green's function shows:
\begin{align*}
\varphi_* \, \sigma(F_U') & \le  \liminf_{n \to \infty} \sum_{\substack{\hat c \in \crit F_{n, U_n} \\ \escaping}} G_{U_n}( p, c) \cdot P_p(\zeta)^{-1} \cdot \delta_{c}
 \\
& \le |(\varphi^{-1})'(\zeta)| \, d\sigma(F')|_{(\partial U \cap \partial \mathbb{D})_{\thick}}.
\end{align*}
The proof is complete.
 \end{proof}

\subsection{Singular parts}

We now investigate the singular part of $\varphi'$ and prove Theorem \ref{nevanlinna-preserving}.

\begin{lemma}
\label{nevanlinna-fact}
 If $f \in \mathcal N$ and $\varphi$ is a holomorphic self-mapping of the unit disk, then $f \circ \varphi \in \mathcal N$.
\end{lemma}

The lemma follows from the description of Nevanlinna class in terms of quotients of bounded analytic functions.

\begin{corollary}
Let $V$ be any simply-connected domain compactly contained in the unit disk. If $\psi' \in \mathcal N$, then $\varphi'  \in \mathcal N$.
\end{corollary}

\begin{proof}
 By the chain rule,
\begin{align}
\label{eq:lower-bound-chain-rule}
F'_U(z)  & = (\psi^{-1})' \bigl (F(\varphi(z)) \bigr ) \cdot F'(\varphi(z)) \cdot \varphi'(z) \\
\label{eq:lower-bound-chain-rule2}
& = \psi' (F_U(z))^{-1} \cdot F'(\varphi(z)) \cdot \varphi'(z).
\end{align}
Since $F'_U(z)$,  $ \psi' (F_U(z))$ and $F'(\varphi(z)) \in \mathcal N$, we must have
 $\varphi' \in \mathcal N$ as well.  
\end{proof}

\begin{lemma}
\label{SSO-radial-limits}
Suppose $f = (S_{\mu_1}/S_{\mu_2})O$, where the measures $\mu_1$ and $\mu_2$ are mutually singular and $O$ is an outer function. The radial limits
\begin{align*}
\lim_{r \to 1} f(r\zeta) = 0, \qquad &\mu_1\text{-a.e.~}\zeta \in \partial \mathbb{D}, \\
\lim_{r \to 1} f(r\zeta) = \infty, \qquad &\mu_2\text{-a.e.~}\zeta \in \partial \mathbb{D}.
\end{align*}
\end{lemma}

The proof is essentially the same as that of \cite[Theorem 1.4]{mashreghi}. 

\begin{lemma}
\label{psi-nonsingular}
Suppose $U \subset \mathbb{D}$ is a simply-connected domain  and $\varphi: \mathbb{D} \to U$ is a conformal map.

{\em (a)} The singular part of $\varphi(z)$ is trivial.

{\em (b)} The singular part of $B(\varphi(z))$ is trivial for any Blaschke product $B$.
\end{lemma}

\begin{proof}
(a) In view of Lemma \ref{SSO-radial-limits}, the lemma is trivial if $0 \notin \partial U$. When $0 \in \partial U$, Beurling's estimate for harmonic measure, e.g. see \cite[Corollary III.9.3]{garnett-marshall}, tells us that
$$
m \bigl (\{\theta: |\varphi({re^{i\theta}})|  < t\} \bigr ) < Ct^{1/2}, \qquad 1/2 < r < 1,
$$
which implies that the functions $\theta \to \log|\varphi({re^{i\theta}})|$ with $1/2 < r < 1$ are uniformly integrable. In particular,
$$
\int_{|z|=r} \log|\varphi(z)| dm \to \int_{|z|=1} \log|\varphi(z)| dm, \qquad \text{as }r \to 1.
$$
By \cite[Lemma 3.1]{inner}, we have $\sigma(\varphi) = 0$ in this case as well.

(b) Write $B$ as a product of M\"obius transformations and notice that $\varphi_m(z) = m(\varphi(z))$ has no singular part for any $m \in \aut(\mathbb{D})$ by part (a).
\end{proof}

\begin{corollary} 
\label{sing-fphiz}
The singular part of $F'(\varphi(z))$ is given by
$$
\sigma(F' \circ \varphi) = \varphi^* \Bigl [ |(\varphi^{-1})'(\zeta)| \cdot d\sigma(F')|_{(\partial U \cap \partial \mathbb{D})_{\thick}} \Bigr ] .
$$
\end{corollary}

\begin{proof} As $F \in \mathscr J$ is an inner function of finite entropy, we may decompose $F' = BSO$ into a Blaschke factor, a singular factor and an outer factor. If $\mu =  \log|F'| dm - \sigma(F')$  then
$$
F'(z) \, = \, B(z)S(z)O(z) \, = \, B(z)  \exp \biggl ( \int_{\partial \mathbb{D}} \frac{\zeta + z}{\zeta - z} d\mu(\zeta) \biggr ).
$$
According to Theorem \ref{radon-nikodym2},
\begin{equation}
\label{eq:lower-bound-composition}
F'(\varphi(z)) = B(\varphi(z)) \exp \biggl ( \int_{\partial \mathbb{D}} \frac{\zeta + z}{\zeta - z} d\mu_U(\zeta) \biggr ),
\end{equation}
where $\mu_U =  C_{\varphi} \mu$. The result follows after taking singular parts and using Lemma \ref{psi-nonsingular}.
\end{proof}

\begin{lemma}
\label{inner-composition}
Suppose $h \in L^1(\partial \mathbb{D})$ and $P_h$ is its Poisson extension to the unit disk. If $F$ is an inner function, then $h \circ F \in L^1(\partial \mathbb{D})$ and $P_h \circ F = P_{h \circ F}$.
\end{lemma}

\begin{proof}
By composing with M\"obius transformations, we may assume that $F(0) = 0$ so that $F$ preserves Lebesgue measure on the unit circle, i.e.~$F^{-1}(E) = E$ for any measurable set $E \subset \partial \mathbb{D}$.

For $N > 0$, form the truncation
$$h_N(\zeta) = \begin{cases}
N, & h(\zeta) \ge N, \\
h(\zeta), & -N \le h(\zeta) \le N, \\
-N, &  h(\zeta) \le -N.
\end{cases}$$
Since bounded harmonic functions are Poisson extensions of their boundary values, $P_{h_N} \circ F = P_{h_N \circ F}$. Since $h_N \to h$ in $L^1(\partial \mathbb{D})$ as $N \to \infty$, it is clear that $P_{h_N}  \to P_h$ and $P_{h_N} \circ F \to P_h \circ F$. The invariance of the Lebesgue measure implies that $h_N \circ F \to h \circ F$ in $L^1(\partial \mathbb{D})$ and $P_{h_N \circ F} \to P_{h \circ F}$.
\end{proof}

\begin{corollary}
\label{psiF-outer}
If $\psi'$ is an outer function, then so is $\psi'(F_U(z))$.
\end{corollary}

 \begin{proof}
Since $\psi'$ is an outer function, $\log | \psi'(z)|$ is the harmonic extension of its boundary values on the unit circle. By Lemma \ref{inner-composition}, this implies that $\log | \psi'(F_U(z))|$ is also the harmonic extension of its boundary values on the unit circle, i.e.~$\psi'(F_U(z))$ is an outer function.
 \end{proof}

Having computed the necessary singular parts, the proof Theorem \ref{nevanlinna-preserving} runs as follows:

 \begin{proof}[Proof of Theorem \ref{nevanlinna-preserving}]
Taking singular parts in (\ref{eq:lower-bound-chain-rule}) and using Corollary \ref{psiF-outer}, we get
\begin{equation}
\label{eq:singular-parts-again}
\sigma(F'_U) = \sigma(F' \circ \varphi) + \sigma(\varphi').
\end{equation}
 From the upper bound in Theorem \ref{radon-nikodym}, which we have proved in Section \ref{sec:upper-bound}, we know that $0 \le \sigma(F'_U) \le \sigma(F' \circ \varphi)$, which tells us that $\sigma(\varphi') \le 0$.
Since $\sigma(F' \circ \varphi)$ is supported on $\varphi^{-1}(\partial U \cap \partial \mathbb{D})_{\thick}$, so must $\sigma(F'_U)$ and $\sigma(\varphi')$. 

In view of Lemma \ref{SSO-radial-limits}, $\sigma(\varphi')$ cannot charge $\varphi^{-1}(\partial U \cap \partial \mathbb{D})_{\thick}$ as $\varphi'$ has a finite non-zero radial limit at any point in this set. Hence, $\sigma(\varphi') = 0$ as desired.
\end{proof}

The above proof shows that when  $V$ is a Nevanlinna domain, $\sigma(F'_U) = \sigma(F' \circ \varphi)$, which was computed in Corollary \ref{sing-fphiz}.
Consequently, Theorem \ref{radon-nikodym} holds for Nevanlinna domains.

\subsection{Lower bound}

In order to establish Theorem \ref{radon-nikodym} for general Jordan domains, we approximate $V$ by smooth Jordan domains from below:

\begin{lemma}
\label{nice-approximations}
Let $V$ be a Jordan domain compactly contained in the unit disk and $U$ be a connected component of $F^{-1}(V)$. There is a sequence of smooth Jordan domains $V_n$ which increase to $V$ such that the conformal maps $\varphi_n: \mathbb{D} \to U_n$ converge to the conformal map $\varphi: \mathbb{D} \to U$ uniformly on the closed unit disk.
\end{lemma}

The following criterion for uniform convergence on the closed unit disk follows from the discussion on \cite[p.~347]{warschawski}.

\begin{lemma}
\label{uniform-convergence-criterion}
Suppose $\{\Omega_n\}_{n=1}^\infty$ is an increasing sequence of Jordan domains whose union $\Omega$ is also a Jordan domain and $p \in \Omega_1$ be a basepoint. For each $n = 1, 2, \dots$, let $\varphi_n: \mathbb{D} \to \Omega_n$ be the conformal map with $\varphi_n(0) = p$ and $\varphi_n'(0) > 0$. Similarly, let $\varphi: \mathbb{D} \to \Omega$ be the conformal map with $\varphi(0) = p$ and $\varphi'(0) > 0$. 
The following statements are equivalent:

\begin{enumerate}
\item The conformal maps $\varphi_n$ converge uniformly on the closed unit disk $\overline{\mathbb{D}}$.
\item The domains $\{ \Omega_n \}$ have a common structural modulus of continuity.
\item There exists a sequence of homeomorphisms $\gamma_n: \partial \mathbb{D} \to \partial \Omega_n$ which converge uniformly to a homeomorphism $\gamma: \partial \mathbb{D} \to \partial \Omega$.
\item The curves $\partial \Omega_n$ converge to $\partial \Omega$ in the Hausdorff topology without backtracking.
\end{enumerate}
\end{lemma}

We now explain the terms in the lemma above.
A {\em crosscut} of a planar domain $\Omega$ is a Jordan arc $\gamma: (0,1) \to \Omega$ such that $\lim_{t \to 0} \gamma(t)$ and $\lim_{t \to 1} \gamma(t)$ exist and belong to $\partial \Omega$. Assuming that $\gamma$ does not pass through the basepoint $p$, we can define $\Gamma$ to be the connected component of $\Omega \setminus \gamma$ which does not contain $p$. In other words, $\Gamma$ is the domain cut off by $\gamma$. Let $\eta: [0, \infty) \to [0, \infty)$ be a non-decreasing function with $\eta(0) = 0$. We say that $\Omega$ has structural modulus of continuity $\eta$ if $$\diam \Gamma \le \eta(\diam \gamma),$$ for any crosscut $\gamma$ of $\Omega$.

We now describe the ``no backtracking'' condition. As is standard in complex analysis, we orient the boundaries of the domains $\partial \Omega_n$, $n=1,2,\dots$ and $\partial \Omega$ counterclockwise. 
We say that the curves $\partial \Omega_n$ converge to $\partial \Omega$ with {\em backtracking} if after passing to a subsequence, there exist two sets of oriented arcs 
$\alpha^1_n, \alpha^2_n: [0, 1] \to \partial \Omega_n$ which converge in the Hausdorff topology to the same arc $\alpha = [a, b] \subset \partial \Omega$ so that
\begin{align*}
\alpha^1_n (0) & \to a, \qquad \alpha^1_n (1) \to b, \\
\alpha^2_n (0) & \to b, \qquad \alpha^2_n (1) \to a.
 \end{align*}
In other words, the arc $\alpha^1_n$ passes by $\alpha$ in one direction, while $\alpha^2_n$ passes by $\alpha$ in the other direction.

\begin{proof}[Proof of Lemma \ref{nice-approximations}]
One particular sequence of domains that works is
$$
V_n = \psi(B(0,1-1/n)), \qquad n = 1, 2, \dots
$$
Since the conformal maps $\psi_n(z) = \psi((1-1/n)z)$ converge uniformly on the closed unit disk to $\psi(z)$, Lemma \ref{uniform-convergence-criterion} tells us that $\partial V_n \to \partial V$ without backtracking.
We claim that $\partial U_n \to \partial U$ also converges without backtracking.

For the sake of contradiction, suppose that there were sequences of arcs $\alpha_n^1, \alpha_n^2 \subset \partial U_n$, which were an instance of backtracking for $\partial U_n \to \partial U$. Since $\partial U \cap \partial \mathbb{D}$ has zero Lebesgue measure, the Hausdorff limit $\alpha = \lim_{n \to \infty} \alpha_n^1 = \lim_{n \to \infty} \alpha_n^2 \subset \partial U$ passes through a point $z_0 \in \mathbb{D}$. As $F$ has at most countably many critical points, we may assume that $z_0 \notin \crit F$.
Pick a small ball $B = B(z_0, r) \subset \mathbb{D}$ on which $F$ is injective. By truncating the arcs $\alpha_n^1, \alpha_n^2$ and dropping finitely many $\alpha_n^i$ with small indices, we may assume that $\alpha_n^1, \alpha_n^2 \subset B(z_0, r/2)$ for all $n$. 
The image curves $F(\alpha_n^1), F( \alpha_n^2) \subset \partial V_n$ provide an instance of backtracking for $\partial V_n \to \partial V$, contradicting the choice of the approximating domains $V_n$.
 \end{proof}

\begin{proof}[Proof of Theorem \ref{radon-nikodym}: Lower bound] 
Suppose $V$ is a general Jordan domain compactly contained in the unit disk. Let $V_1 \subset V_2 \subset \dots$ be an increasing sequence of smooth Jordan domains whose union is $V$, given by Lemma \ref{nice-approximations}. We choose connected components $U_n \subset F^{-1}(V_n)$ so that $U_1 \subset U_2 \subset \dots$ increase to $U$.
By the theorem applied to the $F_{U_n}$, we have
\begin{equation}
\label{eq:Un-RN}
(\varphi_{U_n})_*\, \sigma(F'_{U_n}) = |(\varphi_{U_n}^{-1})'(\zeta)| \cdot d\sigma(F')|_{(\partial U_n \cap \partial \mathbb{D})_{\thick}}.
\end{equation}
In view of the Schwarz lemma, the measures on the right side of (\ref{eq:Un-RN}) increase with $n$. 
Consequently, the component inner functions $F_{U_n}$ converge to $F_U$ in the stable topology and we may use Lemma \ref{uniform-convergence-criterion} to conclude that
\begin{equation}
\label{eq:Un-RN2}
(\varphi_{U})_*\, \sigma(F'_{U}) = |(\varphi_{U}^{-1})'(\zeta)| \cdot d\sigma(F')|_{(\partial U \cap \partial \mathbb{D})_{\thick}}.
\end{equation}
The proof is complete.
\end{proof}

\begin{remark}
The above proof shows that the measure $(\varphi_{U})_*\, \sigma(F'_{U})$ is supported on the set of points $\zeta \in \partial U \cap \partial \mathbb{D}$ for which the radial boundary value $F(\zeta) \in V$. (A priori, the radial boundary value could be in $\overline{V}$.) In particular, $\partial U$ is thin at $\sigma(F')$ a.e.~boundary point $\zeta \in \partial U \cap \partial \mathbb{D}$ with $F(\zeta) \in \partial V$.
\end{remark}

\section{Background on inner functions}
\label{sec:inner}

In this section, we gather some estimates on inner functions of finite entropy. 
Let $\lambda_{\mathbb{D}} = \frac{2}{1-|z|^2}$ be the hyperbolic metric on the unit disk.
A holomorphic self-map $F$ of the unit disk naturally defines the conformal metric
$$
\lambda_F = F^* \lambda_{\mathbb{D}} = \frac{2|F'(z)|}{1-|F(z)|^2}.
$$
With the above definition, if $\gamma \subset \mathbb{D}$ is a rectifiable curve, then the hyperbolic length of $F(\gamma)$ is $\int_\gamma \lambda_F$.

We begin by recalling \cite[Lemma 2.3]{inner} which describes $\lambda_F$ as the minimal conformal pseudometric of curvature $-1$ that lies above $|\inn F'| \lambda_{\mathbb{D}}$. 
To state the lemma in a more pedestrian way, we write $B_C$ for the Blaschke product with zero set $C$.

\begin{lemma}[Fundamental lemma]
\label{fundamental-lemma}
For any inner function $F \in \mathscr J$,
\begin{equation}
\label{eq:fundamental-lemma}
\lambda_F \ge |\inn F'| \lambda_{\mathbb{D}}.
\end{equation}
Suppose $F_1, F_2 \in \mathscr J$ are two inner functions with $\inn F'_1 = B_{C_1} S_{\mu_1}$ and $\inn F'_2 = B_{C_2} S_{\mu_2}$. If $C_1 \subseteq C_2$ and $\mu_1 \le \mu_2$ then
\begin{equation}
\label{eq:fundamental-lemma2}
\lambda_{F_1} \ge \lambda_{F_2}.
\end{equation}
\end{lemma}

In Lemmas \ref{stable-bound}--\ref{jensen-type} below, $\mu$ will be a measure supported on a Beurling-Carleson set $E$. We write $F_\mu \in \mathscr J$ for the inner function with $\inn F_\mu'=S_\mu$, $F_\mu(0)=0$ and $F'_\mu(0) > 0$, given by Theorem \ref{main-thm-inner}. The following coarse estimate can be found in   \cite[Corollary 3.4]{stable}:

\begin{lemma}
\label{stable-bound}
 The inner function $F_\mu$ extends analytically past any
arc $I \subset \partial \mathbb{D}$ which does not meet the support of $\mu$.
The derivative of $F_\mu$ on the unit circle is bounded by
\begin{equation}
\label{eq:stable-bound}
|F'_{\mu}(\zeta)| \le C \bigl (\mu(\partial \mathbb{D}) \bigr ) \cdot \dist(\zeta, E)^{-4}, \qquad \zeta \in \partial \mathbb{D},
\end{equation}
where $C(t)$ is a positive increasing function defined on $(0, \infty)$.
\end{lemma}

The above estimate will suffice to study radial limits of $F_\mu$. For thick limits, we require a more refined estimate:

\begin{lemma}
\label{refined-bound}
Let $\zeta \in \partial \mathbb{D} \setminus E$ be a point on the unit circle, not contained in $E$. Consider the point $z = (1-\delta)\zeta$ where $\delta = \dist(\zeta, E)$. If $\delta < 1$ and $P_\mu(z) \ge 1$ then
\begin{equation}
\label{eq:refined-bound}
|F'_{\mu}(\zeta)| \le C \cdot \frac{P_\mu(z)}{\delta},
\end{equation}
for some universal constant $C > 0$.
\end{lemma}

To show Lemma \ref{refined-bound}, we will use the following lemma:

\begin{lemma}
\label{mcmullen-injectivity}
Let $F$ be a holomorphic self-map of the disk, $\zeta \in \partial \mathbb{D}$ be a point on the unit circle and $0 < \rho < 1$.
There exists a universal constant $\varepsilon_0 > 0$ so that the inequality 
$$
\lambda_F > (1 - \varepsilon_0)\lambda_{\mathbb{D}}, \qquad \text{on }B(\zeta, \rho) \cap \mathbb{D},
$$
implies that $F$ is injective on $B(\zeta, \rho)$.
\end{lemma}

\begin{proof}
{\em Step 1.} We first show that the assumption implies that at any point $z \in B(\zeta, \rho) \cap \mathbb{D}$, the 2-jet of $F$  matches the
2-jet of a hyperbolic isometry with an error of $O(\varepsilon_0)$.

Composing with M\"obius transformations, one may assume that $z = 0$, $F(0) = 0$ and $0 \le F'(0) < 1$. With this normalization, the assumption says that $F'(0) > 1 - \varepsilon_0$. We need to show that $|F''(0)| = O(\varepsilon_0)$. Applying the Schwarz lemma to $F(w)/w$ shows that the hyperbolic distance
$$d_{\mathbb{D}} \bigl (F(w)/w, F'(0) \bigr ) = O(1), \qquad \text{for }w \in B(0, 1/2).$$
Taking note of the location of $F'(0) \in \mathbb{D}$, this implies that $|F (w) - w| = O(\varepsilon_0)$ for $w \in B(0, 1/2)$. Cauchy's integral formula now gives the estimate for the second derivative. 

\medskip

{\em Step 2.}
As a result, when $\varepsilon_0 > 0$ is sufficiently small, the image of a hyperbolic geodesic $\gamma$ contained in $B(\zeta,\rho) \cap \mathbb{D}$ will have hyperbolic geodesic curvature less than $1$ (the curvature of a horocycle). Since such a curve cannot cross itself, $F$ is injective on $\gamma$.

\medskip

{\em Step 3.}
As $B(\zeta, \rho) \cap \mathbb{D}$ is convex in the hyperbolic metric, this shows that $F$ is injective on $B(\zeta, \rho) \cap \mathbb{D}$.  Taking advantage of the fact that $F$ is an inner function, one may use the aforementioned injectivity together with \cite[Theorem 2.6]{mashreghi} to see that $F$ extends analytically across $B(\zeta, \rho) \cap \partial \mathbb{D}$. By the Schwarz reflection principle, $F$ extends to a meromorphic function on $B(\zeta, \rho)$.
Finally, as $F$ is symmetric with respect to the unit circle, $F$ is injective on the whole ball $B(\zeta, \rho)$. 
\end{proof}

\begin{remark}
Steps 1 and 2 of the above proof are taken from \cite[Corollary 10.7 and Proposition 10.9]{mcmullen-rtree}.
\end{remark}

\begin{proof}[Proof of Lemma \ref{refined-bound}]
Consider the ball
$$
B = B (\zeta, \rho ), \qquad \rho =  \frac{c_1 \cdot \delta}{P_\mu(z)}, \qquad 0 < c_1 < 1.
$$
From the assumption, it follows that $0 < \rho < \delta$, so that $\overline{B(\zeta, \rho)} \cap \supp \mu = \emptyset$. In particular, $F_\mu$ extends meromorphically to $B(\zeta, \rho)$ by Schwarz reflection and Lemma \ref{stable-bound}. By requesting $c_1 > 0$ to be sufficiently small, we can make $P_\mu(z)$ to be as small as we wish on $B (\zeta, \rho)$.
Since $|S_\mu(z)| = \exp(- P_\mu(z))$, we can choose $c_1 > 0$ so that $|S_\mu(z)| > 1 - \varepsilon_0$  on $B (\zeta, \rho)$, where $\varepsilon_0$ is the constant from Lemma \ref{mcmullen-injectivity}. 

The estimate (\ref{eq:fundamental-lemma}) and Lemma \ref{mcmullen-injectivity} guarantee that $F_\mu$ is injective on  $B (\zeta, \rho)$.
A compactness argument shows that
 $$
\biggl | \frac{F_\mu'(z)}{F_\mu'(\zeta)} - 1 \biggr | < 0.1, \qquad z \in B(\zeta, c_2 \cdot \rho),
$$
where $0 < c_2 < 1$ is a universal constant. If the desired estimate (\ref{eq:refined-bound}) were false, then
$F_\mu((1-c_2\rho)\zeta) \notin \mathbb{D}$, which would contradict that $F_\mu$ is an inner function.
\end{proof}

To estimate the derivative of $F_\mu$ in the unit disk, we will use a Jensen-type formula for Nevanlinna functions, e.g.~see \cite[Lemma 3.1]{inner}:
\begin{lemma}
\label{jensen-type}
For $z \in \mathbb{D}$, we have:
\begin{equation}
\label{eq:jensen-type}
\log|F_{\mu}'(z)|=\int_{\partial\mathbb{D}}\log|F_{\mu}'(\zeta)|d\omega_{z}(\zeta)-\int_{\partial\mathbb{D}}P_{z}(\zeta)d\mu(\zeta),
\end{equation}
where $\omega_z$ denotes the harmonic measure on the unit circle as viewed from $z$.
\end{lemma}

The normalization $F_\mu(0)=0$ guarantees that $|F'_\mu(\zeta)| \ge 1$ on the unit circle, e.g.~see \cite[Theorem 4.15]{mashreghi}. In particular, the first term in (\ref{eq:jensen-type}) is positive, while the second term is negative.

We now give a short glimpse of how the above techniques will be used to show the abundance of radial and thick limits in Sections \ref{sec:abundance-radial} and \ref{sec:abundance-thick}. In order for the estimates of Lemmas \ref{stable-bound}, \ref{refined-bound} and \ref{jensen-type}
to be effective, the point $z \in \mathbb{D}$ must be far away from the support of $\mu$. To obtain estimates near the support of $\mu$, we use the Fundamental Lemma 
(Lemma \ref{fundamental-lemma}) to remove the part of the support of $\mu$ which obstructs the vision of $z$. However, we do not want to remove too much mass to retain enough hyperbolic contraction of $F_\mu$. These conflicting demands make the arguments below somewhat delicate.

\section{Background on measures}
\label{sec:measures}

In this section, we discuss several facts about measures that are supported on Beurling-Carleson sets.

We begin by introducing some notation. For a point $x \in \partial \mathbb{D}$ and $0 < \varepsilon < \pi$, we let $I(x, \varepsilon) = (e^{-i\varepsilon}x ,e^{i\varepsilon}x)$ denote the arc of the unit circle centered at $x$ of length $2\varepsilon$. We denote the left and right arcs by
 $I^L(x,\varepsilon) = (e^{-i\varepsilon} x , x]$ and $I^R(x,\varepsilon) = [x, e^{i\varepsilon} x)$ respectively.
 For a measure $\mu$ on the unit circle, we write $\mu (x,\varepsilon) = \mu(I(x,\varepsilon))$. The quantities $\mu ^L (x,\varepsilon)$ and $\mu ^R (x,\varepsilon)$ are defined similarly.

\subsection{Non-centered maximal function}
\label{sec:ncf}

We consider an analogue of the Hardy-Littlewood maximal function for measures. For two measures $\mu, \nu$ on $\partial \mathbb{D}$, the {\em non-centered maximal function} is given by
\begin{equation}
H_\mu \nu(x) = \liminf_{x\in I,\,  |I|\to 0} \, \frac{\mu(I)}{\nu(I)}.
\end{equation}
\begin{lemma}
\label{uncentered lemma}
Let $\mu, \nu$ be finite measures on $\partial \mathbb{D}$ such that $\nu$ is outer regular and $\nu(\supp \mu)=0$. For a.e.~$x$ with respect to $\mu$,
$$
H_\mu \nu(x)=\infty.
$$
\end{lemma}

\begin{proof}
For $N > 0$, consider the set
$$
A_N := \bigl \{x\in \supp \mu : H_\mu \nu(x) < N \bigr \}.
$$
For every point $x \in A_N$, there is a sequence of  arcs $\{ I_n^x \}_{n=1}^\infty$ such that $x\in I_n^x$, $|I_n^x|\to0$ as $n\to \infty$ and $\mu(I_n^x)<N\nu(I_n^x)$.

As $A_N$ is contained in the support of $\mu$, we have $\nu(A_N)=0$ by assumption. 
Since $\nu$ is outer regular, for any $\varepsilon>0$, there exists an open set $C_\varepsilon \supset A_N$ with $\nu(C_\varepsilon)<\varepsilon$.
Let 
$$
\mathcal{A}_{N,\varepsilon} := \bigl \{I_n^x\subseteq C_\varepsilon \,: \,  x\in A_N,\, n \ge 1 \bigr \}.
$$
Clearly, the arcs in $\mathcal{A}_{N,\varepsilon}$ cover the set $A_N$.  As explained in \cite{aldaz}, one can find a disjoint sub-collection of arcs $\Gamma \subset \mathcal{A}_{N,\varepsilon}$ such that
$$
\mu(\cup \mathcal{A}_{N,\varepsilon})\le 3\mu (\cup \Gamma),
$$
where $\cup \mathcal{A}_{N,\varepsilon}$ and $\cup \Gamma$ denote the union of arcs in
$\mathcal{A}_{N,\varepsilon}$ and $\Gamma$ respectively.
Therefore,
$$
\mu(A_N)\leq \mu(\cup \mathcal{A}_{N,\varepsilon}) \le 3\mu(\cup \Gamma)\le 3 N \cdot \nu(\cup \Gamma)\le 3N\varepsilon.
$$
Taking $\varepsilon \to 0$, we see that $\mu (A_N) = 0$, from which the result follows.
\end{proof}

To a Beurling-Carleson set $E$, we associate the {\em auxiliary measure} 
\begin{equation}
\label{eq:auxiliary-measure}
\nu_E = \log^+ \frac{1}{d(\zeta,E)} \, dm
\end{equation}
on the unit circle. The summability condition (\ref{eq:beurling-carleson}) ensures that $\nu_E$ is a finite outer regular measure. It is not difficult to see that 
\begin{equation}
\label{eq:nue-lowerbound}
\nu_E(I) \ge C |I| \log \frac{1}{|I|}, \qquad I \subset \partial \mathbb{D}, \qquad \overline{I} \cap E \ne \emptyset,
\end{equation}
where $C$ is a universal constant, independent of $E$. More precisely, we may write $I \setminus E = \bigcup J_k$ as a countable union of complementary arcs.
Since each complementary arc $J_k$ has at least one endpoint in $E$,
$$\nu_E(J_k) \, \ge \, C |J_k| \log \frac{1}{|J_k|} \, \ge \, C |J_k| \log \frac{1}{|I|}.$$
From here, (\ref{eq:nue-lowerbound}) follows after summing over $k$. Applying Lemma \ref{uncentered lemma} with  $\nu = \nu_E$ shows the following corollary:

\begin{corollary}
\label{Shapiro bound}
Let $\mu$ be a finite measure supported on  Beurling-Carleson set. For $\mu$ a.e.~$x \in \partial \mathbb{D}$,
$$
\lim_{\varepsilon \to 0} \frac{\mu ^R(x,\varepsilon)}{\nu^R(x,\varepsilon) }=\infty
\qquad
\text{and}
\qquad
\lim_{\varepsilon \to 0} \frac{\mu ^R(x,\varepsilon)}{\varepsilon \log  \frac{1}{\varepsilon} }=\infty.
$$
Similar statements hold for left intervals and centered intervals.
\end{corollary}

\subsection{A rearrangement inequality}


\begin{lemma}
\label{rearrangement lemma for Poisson extension}
Suppose $\mu, \nu$ are two measures on $\partial \mathbb{D}$ and $x \in \partial \mathbb{D}$ is a point. If
$$
\mu(I(x,\varepsilon)) \ge \nu(I(x,\varepsilon)), \qquad 0 < \varepsilon \le \pi,
$$
 then their Poisson extensions satisfy
$$
P_{\mu} (rx)\ge P_{\nu} (rx),
$$
for any $0 < r<1$.
\end{lemma}

\subsection{Local behaviour}

\begin{theorem}
\label{local-behaviour}
Suppose $\mu$ is a singular measure supported on a Beurling-Carleson set.  For $\mu$-a.e.~$x \in \partial \mathbb{D}$, the integral
\begin{equation}
\label{eq:local-behaviour}
\int_0^1 \mu(x, \varepsilon)^{-1} d\varepsilon 
\end{equation}
is finite.
\end{theorem}

\begin{proof}
Since $\mu$ is a singular measure, for $\mu$-a.e.~$x \in \partial \mathbb{D}$, $\lim_{\varepsilon \to 0} \frac{\mu(x, \varepsilon)}{\varepsilon} = \infty$.
To prove the lemma, we will show that the double integral
$$
\int_{E} \int_0^1 \mu(x, \varepsilon)^{-1} \, d\varepsilon d\mu(x) < \infty.
$$
For a point $x \in \partial \mathbb{D}$, we write $S(x)$ for the Stolz angle of opening $\pi/2$ with vertex at $x$ and $K_E$ for the union of the Stolz angles emanating from points $x \in E$. It is not difficult to see that
$$
1 + \sum_I |I| \log^+ \frac{1}{|I|} \asymp \int_{K_E} \frac{dA(z)}{1-|z|},
$$
where $I$ ranges over the complementary arcs in $\partial \mathbb{D} \setminus E$. 
We subdivide the above integral over individual Stolz angles:
$$
\int_{K_E} \frac{dA(z)}{1-|z|} = \int_E \int_{S(\zeta)} \eta(z) \cdot \frac{dA(z)}{1-|z|} d\mu(\zeta),
$$
where the function $\eta(z) = \mu(I_z)^{-1}$ measures how many Stolz angles cover $z$. Here, 
 $I_z$ is the arc of the unit circle that consists of points $\zeta$ for which $z \in S(\zeta)$.
From
\begin{align*}
\int_{S(\zeta) \cap \{1-|z| = \varepsilon\}} \eta(z) \cdot \frac{|dz|}{1-|z|} & \ge  \min_{z \in S(\zeta) \cap \{1-|z| = \varepsilon\}} \mu(I_z)^{-1} \\
& \ge \mu(\zeta, 3\varepsilon)^{-1},
\end{align*}
we deduce that
$$
\int_{E} \int_0^1 \mu(\zeta, 3\varepsilon)^{-1} d\varepsilon d\mu(\zeta) \lesssim 1 + \sum_I |I| \log^+ \frac{1}{|I|} 
$$
as desired.
\end{proof}

\begin{remark}
Curiously enough, the integral (\ref{eq:local-behaviour}) also appears in the study of harmonically weighted Dirichlet spaces, see \cite{el-fallah, el-fallah2}.
\end{remark}

 \section{Some thick regions}
\label{sec:thick-poisson-regions}

In this section, we prefer to work with domains defined in the upper half-plane $\mathbb{H}$.
For a continuous function $f: \mathbb{R} \to [0, \infty)$ with $f(0) = 0$ and $f(x) \le |x|/4$, we define an approach region $\mathcal U_f \subset \mathbb{H}$ by
$$
 \mathcal U_f = \bigl \{ x+iy \in \mathbb{H} : y > f(x) \bigr \}.
 $$
One may then investigate the thickness of $\mathcal U_f$ at $0 \in \partial \mathcal U_f \cap \partial \mathbb{H}$ using the auxiliary regions 
\begin{align*}
\mathcal U_{f^+} = \mathbb{H}
&  \setminus \bigcup_{x_0 \in \mathbb{R}} \bigl [x_0 - f(x_0), x_0 + f(x_0) \bigr ] \times \bigl [0, 2 f(x_0) \bigr ]
\end{align*}
and
\begin{align*}
\mathcal U_{f^-} = \mathbb{H}
&  \setminus \bigcup_{x_0 \in \mathbb{R}} \biggl [x_0 - \frac{f(x_0)}{4}, \ x_0 +  \frac{f(x_0)}{4} \biggr ] \times \biggl [0, \frac{f(x_0)}{2} \biggr ].
\end{align*}

While the above regions are slightly different from the ones obtained by than taking logarithms, forming the analogous regions in $\mathcal S$ and exponentiating back to $\mathbb{H}$, the difference is not essential as $f(x) \le |x|/4$. The assumption $f(x) \le |x|/4$ does not cost us anything since thickness is a local property, while possessing an inner tangent at 0 necessitates that $\lim_{x \to 0^+} f(x)/x = 0$. 
 
\begin{theorem}
\label{H-thickness-condition}
Suppose $\mu$ is a finite measure on the real line $\mathbb{R}$ such that
\begin{equation}
\label{eq:H-thickness-condition}
\int_0^1 \mu(0, x)^{-1} \,dx < \infty,
\end{equation}
where $\mu(0, r) = \mu(B(0,r))$.
For any $c > 0$, the symmetric region $\mathcal U_f$ with
$$
f(\pm x) = c \cdot \frac{x^2}{\mu(0, x/2)}, \qquad x > 0,
$$
 is thick at $0$.
\end{theorem}

\begin{proof}
Form the symmetric approach region $\mathcal U_g$ with
 $$
g(\pm x) = 2c \cdot \frac{x^2}{\mu(0, x/4)}, \qquad x > 0.
$$
 Since $g(x) \ge f(x)$, we have $\mathcal U_g \subset \mathcal U_f$. The crude estimate 
 $$
g^{+}(x) \, \ge \, \min_{x_0 \in [x/2, 2x]} 2 g(x_0) \, \ge \, 4c \cdot  \frac{(x/2)^2}{\mu(0, x/2)} \, = \, f(x),
 $$
for $x \ge 0$, shows that $\mathcal U_{g^+} \subset \mathcal U_f$. Since
$$
\int_0^1 \frac{g^+(x)}{x^2} dx \, \lesssim \, \int_0^1 \frac{f(x)}{x^2} dx \, < \, \infty,
$$
$\mathcal U_g$ is thick at 0. It follows that the bigger region $\mathcal U_f$ is thick at 0 as well.
\end{proof}

\begin{corollary}
If $\mu$ is a finite measure on the real line $\mathbb{R}$ which satisfies (\ref{eq:H-thickness-condition}), then for any $\eta > 0$, the region
 $$
 \Omega_\eta = \{ z \in \mathbb{D} : P_\mu(z) > \eta \}
 $$
 is thick at 0. (If $\Omega_\eta$ is disconnected, then we take the connected component which has an inner tangent at 0.) 
\end{corollary}

To see the corollary, notice that $\Omega_\eta$ contains the symmetric region $\mathcal U_f$ from Theorem \ref{H-thickness-condition} for some $c = c(\eta) > 0$.

\section{Abundance of radial limits}
\label{sec:abundance-radial}

Given a point $x \in \partial \mathbb{D}$, the diameter that passes through $x$ divides $\partial \mathbb{D}$ into a left arc $I^L$ and a right arc $I^R$. Even though $I^L$ and $I^R$ depend on $x$, we will suppress $x$ from the notation. For a measure $\mu$, we set $\mu^L = \mu|_{I^L}$ and $\mu^R = \mu|_{I^R}$. We view $I^L$ and $I^R$ as closed arcs, so that $\mu(\{x\})$ and $\mu(\{-x\})$ appear in both $\mu^L$ and $\mu^R$.
 
\begin{lemma}
\label{contraction-bound}
Suppose $\mu$ is a singular measure supported on a Beurling-Carleson set $E$ and $F_\mu \in \mathscr J$ is an inner function of finite entropy such that $\inn F'_\mu = S_\mu$. For a.e.~$x \in \partial \mathbb{D}$ with respect to $\mu$ and any real number $N \ge 1$,
\begin{equation}
\label{eq:contraction-bound}
\lambda_{F_{\mu^L}} \bigl ((1-\varepsilon)x \bigr ) \lesssim \varepsilon^{N},
\qquad \lambda_{F_{\mu^R}}\bigl ((1-\varepsilon)x \bigr ) \lesssim \varepsilon^{N},
\end{equation}
for all $0 < \varepsilon < \varepsilon_0(x, N)$ sufficiently small.
\end{lemma}

\begin{proof}[Proof of Lemma \ref{contraction-bound}]

{\em Step 0.} For brevity, we write $z = (1-\varepsilon)x$. Without loss of generality, we work with the measure $\mu^L$ to the left of $x$.
Form the auxiliary measure
$$
\nu^L = \log^+ \frac{1}{d(\zeta,E \cap I^L)} \, dm.
$$
As the measures $\mu^L,\nu^L$ satisfy the conditions of Corollary \ref{Shapiro bound}, for $\mu$ a.e.~$x \in \partial \mathbb{D}$ and any real constants $C_1, C_2 > 0$ (to be chosen), there exists an $\varepsilon_0 > 0$ so that
$$
\frac{\mu^L(x,\varepsilon)}{\nu^L(x, \varepsilon)} > C_1
\qquad
\text{and}
\qquad
\frac{\mu^L(x, \varepsilon)}{\varepsilon \log\frac{1}{\varepsilon}} > C_2,
$$
for all $0 < \varepsilon < \varepsilon_0$.

\medskip

{\em Step 1.} 
 By Lemma \ref{jensen-type} and the estimate (\ref{eq:stable-bound}), we have
\begin{align*}
\log|F_{\mu^L}'(z)| & =\int_{\partial\mathbb{D}}\log|F_{\mu^L}'(\zeta)|d\omega_{z}(\zeta)-\int_{\partial\mathbb{D}}P_{z}(\zeta)d\mu^L(\zeta) \\
& \le C + 4\int_{\partial\mathbb{D}}\log^+ \frac{1}{d(\zeta,E \cap I^L)}d\omega_{z}(\zeta)-\int_{\partial\mathbb{D}}P_{z}(\zeta) d\mu^L(\zeta) \\
& = C + 4\int_{\partial\mathbb{D}}P_{z}(\zeta)d\nu^L(\zeta)-\int_{\partial\mathbb{D}}P_{z}(\zeta) d\mu^L(\zeta).\\
& \le C + \int_{I^L}P_{z}(\zeta)(4d\nu^L(\zeta) - d\mu^L(\zeta)) + 4\int_{I^R}P_{z}(\zeta)d\nu^L(\zeta).
\end{align*}
Since $z$ belongs to the line segment $[0, x]$,
$$
\log|F_{\mu^L}'(z)|  \le C + \int_{I^L}P_{z}(\zeta)(8 \, d\nu^L(\zeta) - d\mu^L(\zeta)).
$$
If $C_1 > 16$, then Lemma \ref{rearrangement lemma for Poisson extension} tells us that
\begin{equation}
\label{eq:bound-on-the-derivative}
\log|F_{\mu^L}'(z)| \le C' - \frac{1}{2} \int_{I^L(x, \varepsilon)} P_{z}(\zeta) d\mu^L(\zeta).
\end{equation}

{\em Step 2.} From (\ref{eq:bound-on-the-derivative}), it is readily seen that if the constant $C_2 > 0$ from Step 0 is sufficiently large, then
$$
\log|F_{\mu^L}'(z)| \le-(N+1)\cdot\log\frac{1}{\varepsilon} \quad \implies \quad |F_{\mu^L}'(z)|\le\varepsilon^{N+1}.
$$
By the Schwarz lemma,
$$
\lambda_{F_{\mu^R}}(z) 
\, = \,
 \frac{2| F_{\mu^L}'(z)|}{1-|F_{\mu^L}(z)|^{2}} \, \le \, 
\frac{2| F_{\mu^L}'(z)|}{1-|z|^2} \, \lesssim \,
 \varepsilon^N
$$
as desired.
\end{proof}

\begin{proof}[Proof of Theorem \ref{main-thm}: Radial limits]
Let $F\in \mathscr{J}$ be an inner function of finite entropy. Suppose that $\inn F' = B S_\mu$. 
We need to show that for a.e.~$x$ with respect to $\mu$, the radial limit $\lim _{r\to1} F(rx)$ exists and lies inside the unit disk. 
By Theorem \ref{main-thm-inner}, we may write
$$
\mu=\sum_{i=1}^{\infty}\mu_i,
$$
where each measure $\mu_i$ is supported on a Beurling-Carleson set. In view of Lemmas \ref{fundamental-lemma} and \ref{contraction-bound}, for a.e.~$x$ with respect to $\mu_i$, we have
$$
\lambda_{F}(rx) \, \le \, \lambda_{F_{\mu_i^L}}(rx) \, \le \, C(x)(1-r), \qquad \text{as } r \to 0.
$$
The last statement implies that the hyperbolic length of $F([0, x])$ is finite. In particular, the radial limit
$F(x) := \lim_{r \to 1} F(rx)$ exists and is contained in $\mathbb{D}$. The proof is complete.
\end{proof}

\begin{remark}
The above argument can be easily adapted to show the abundance of non-tangential limits.
\end{remark}

\section{Abundance of thick limits}
\label{sec:abundance-thick}

Without loss of generality, we may investigate the behaviour of $F_\mu \in \mathscr J$ to the right of $x \in \partial \mathbb{D}$, as the behaviour of $F_\mu$ to the left of $x$ can be examined in an analogous manner.

For a point $x \in \partial \mathbb{D}$ and $\varepsilon >0$, we write $x_\varepsilon =  e^{i\varepsilon} x$ and $z_{\varepsilon, h}= e^{i\varepsilon} (1-h)x$. For a singular measure $\mu$, we set $\mu_{\varepsilon} =\mu|_{I^L\cup {I^R(x,\varepsilon)}}$ and $h_\varepsilon = \varepsilon /P_{\mu_\varepsilon}(z_{3\varepsilon, \varepsilon})$.
By Lemma \ref{refined-bound}, we have:

\begin{lemma}
\label{strong-boundary-bound}
Suppose $\mu$ is a singular measure supported on a Beurling-Carleson set $E$. For a.e.~$x \in \partial \mathbb{D}$ with respect to $\mu$,
\begin{equation}
\label{eq:F_mu and S_mu'-bound}
|F_{\mu_\varepsilon}'(\zeta)| \, \lesssim \, \frac{P_{\mu_\varepsilon}(z_{3\varepsilon, \varepsilon})}{\varepsilon} \, = \, \frac{1}{h_\varepsilon}, \qquad \zeta \in I(x_{3\varepsilon},\varepsilon),
\end{equation}
for any $0 < \varepsilon < \varepsilon_0(x)$ sufficiently small.
\end{lemma}

An argument involving Lemmas \ref{jensen-type} and \ref{rearrangement lemma for Poisson extension} shows:

\begin{corollary}
\label{contraction at thick point}
Suppose $\mu$ is a singular measure supported on a Beurling-Carleson set $E$. For a.e.~$x \in \partial \mathbb{D}$ with respect to $\mu$ there exists $\varepsilon_0>0$ such that
$$
|F_{\mu_\varepsilon}'(z_{3\varepsilon, h})|\le  \frac{K}{h_\varepsilon} \cdot e^{-\frac{h}{8 h_\varepsilon}},
$$
for any $0<C h_\varepsilon < h <\varepsilon<\varepsilon_0$. In particular,
\begin{equation}
\label{eq:contraction at thick point}
\int_{C \cdot h_{\varepsilon}}^{\varepsilon}|F_{\mu_{\varepsilon}}'(z_{3\varepsilon ,h} )| \, dh  \le 8K \cdot e^{-\frac{C}{8}}.
\end{equation}
\end{corollary}

We follow the same strategy as in the proof of Lemma \ref{contraction-bound}. In the proof below, we will frequently re-use the constants $C, C', K > 0$.

\begin{proof}
{\em Step 0.} 
Form the auxiliary measures
$$
\nu_L = \log^+ \frac{1}{d(\zeta,E \cap I^L)} \, dm, \qquad \nu_R = \log^+ \frac{1}{d(\zeta,E \cap I^R)} \, dm
$$
and
$$
\nu_{\varepsilon} = \log^+ \frac{1}{d \bigl (\zeta,E \cap (I^L \cup I^R(x,\varepsilon)) \bigr )} \, dm.
$$
As the pairs of measures $(\mu^L,\nu^L)$ and $(\mu^R,\nu^R)$ satisfy the conditions of Corollary \ref{Shapiro bound}, for $\mu$ a.e.~$x \in \partial \mathbb{D}$ and any real constants $C_1, C_2 > 0$ (to be chosen), there exists an $\varepsilon_0 > 0$ so that
$$
\frac{\mu^L(x,\varepsilon)}{\nu^L(x, \varepsilon)} > C_1,
\quad
\frac{\mu^R(x,\varepsilon)}{\nu^R(x, \varepsilon)} > C_1,
\quad  
\frac{\mu^L(x, \varepsilon)}{\varepsilon \log\frac{1}{\varepsilon}} > C_2,
\quad
\frac{\mu^R(x, \varepsilon)}{\varepsilon \log\frac{1}{\varepsilon}} > C_2,
$$
for all $0 < \varepsilon < \varepsilon_0$.

\medskip

{\em Step 1.}  By Lemma \ref{jensen-type}, we have
\begin{equation}
\label{eq:jensen-thick}
\log|F_{\mu_\varepsilon}'(z_{3\varepsilon,h})| = \int_{\partial\mathbb{D}} \log|F_{\mu_\varepsilon}'(\zeta)| d\omega_{z_{3\varepsilon,h}}(\zeta) -\int_{I^L \cup I^R(x, \varepsilon)}P_{z_{3\varepsilon,h}}(\zeta) d\mu_\varepsilon(\zeta).
\end{equation}
We split the first integral in the above equation over four arcs:
\begin{itemize}
\item $I_1=I^L\cup I^R(x, \varepsilon)$, 

\item$I_2 = (x_\varepsilon, x_{2\varepsilon})$,

\item$I_3 = (x_{2\varepsilon},x_{4\varepsilon})$,

\item$I_4 = I^R \setminus (x, x_{4\varepsilon})$.
\end{itemize}
Since the second integral in (\ref{eq:jensen-thick}) is supported only on $I_1$, we leave it as it is.

\medskip

{\em Step 2.}
By Lemma \ref{stable-bound}, the contribution  of the arc $I_1$ to the right side of (\ref{eq:jensen-thick}) is bounded above by
$$
C + \int_{I^L \cup I^R(x, \varepsilon)}P_{z_{3\varepsilon,h}}(\zeta)(4d\nu_\varepsilon(\zeta) - d\mu_\varepsilon(\zeta)).
$$
We claim that this expression is
$$
\le \, C' - \frac{1}{2} \int_{I^L \cup I^R(x, \varepsilon)}P_{z_{3\varepsilon,h}}(\zeta)  d\mu_\varepsilon(\zeta) \, = \, C' -\frac{1}{2} \cdot P_{\mu_\varepsilon}(z_{3\varepsilon,h}),
$$
provided that $0 < \varepsilon < \varepsilon_0$ is sufficiently small.

The bound on $I^L$ follows from Lemma \ref{rearrangement lemma for Poisson extension} if $C_1 > 8$. To see the corresponding bound on $I^R(x,\varepsilon)$, one only needs to take $C_1$ sufficiently large and use the fact that the Poisson kernel is pinched:
 $$
1/K \le \frac{ P_{z_{3\varepsilon,h}}(\zeta_1)}{ P_{z_{3\varepsilon,h}}(\zeta_2)} \le K, \qquad \zeta_1,\zeta_2 \in I^R(x,\varepsilon).
 $$

\medskip

{\em Step 3.}
We also use the coarse bound from Lemma \ref{stable-bound} to estimate the integrals over the arcs $I_2$ and $I_4$:
\begin{align*}
\int_{I_2 \cup I_4}\log|F_{\mu_\varepsilon}'(\zeta)|d\omega_{z_{3\varepsilon,h}}(\zeta)
 & \le C + 4 \int_{I_2 \cup I_4} \log^+ \frac{1}{d(\zeta,x_\varepsilon)} \, P_{z_{3\varepsilon,h}}(\zeta) dm(\zeta) \\
& \le C + \frac{K h}{\varepsilon}\log \frac{1}{\varepsilon},
\end{align*}
while we estimate the integral over $I_3$ using the refined estimate from Lemma \ref{strong-boundary-bound}:
$$
\int_{I_3}\log|F_{\mu_\varepsilon}'(\zeta)|d\omega_{z_{3\varepsilon,h}}(\zeta) \, \le \, \log\frac{K}{h_\varepsilon}
\, = \, C + \log\frac{1}{h_\varepsilon}.
$$

\smallskip

{\em Step 4.} Putting the estimates from Steps 2 and Steps 3 together, we arrive at
\begin{equation}
\label{eq:equations-together}
\log|F_{\mu_\varepsilon}'(z_{3\varepsilon, h})| \, \le \,  
C \, + \,  \log \frac{1}{h_\varepsilon} \, + \, \frac{K h}{\varepsilon}\log \frac{1}{\varepsilon} 
\, - \, \frac{1}{2} \cdot P_{\mu_\varepsilon}(z_{3\varepsilon,h}).
\end{equation}
To clean up the right hand side of the above equation, we choose the constant $C_2$ from Step 0 sufficiently large so that
$$
\frac{\mu^R(x, \varepsilon)}{\varepsilon \log\frac{1}{\varepsilon}} \, > \,  C_2 \, > \, 40 \, K,
$$
in which case,
$$
\frac{1}{4} \cdot P_{\mu_\varepsilon}(z_{3\varepsilon,h}) \, \ge \, \frac{1}{4} \cdot \frac{\mu^R(x, \varepsilon)h}{10 \, \varepsilon^2} \, \ge \, \frac{Kh}{\varepsilon} \log \frac{1}{\varepsilon}.
$$
Since
$$
\frac{h}{8h_\varepsilon} \, = \,  \frac{h}{8\varepsilon} \cdot P_\mu(z_{3\varepsilon}, \varepsilon) \, \le \, \frac{1}{4} \cdot P_\mu(z_{3\varepsilon, h}),
$$
(\ref{eq:equations-together}) simplifies to
$$
\log|F_{\mu_\varepsilon}'(z_{3\varepsilon, h})|\le  C + \log \frac{1}{h_\varepsilon} -\frac{h}{8 h_\varepsilon}.
$$
The lemma follows after exponentiating both sides.
\end{proof}

With the above estimate at our disposal, we can prove Theorem \ref{main-thm2}:

\begin{proof}[Proof of Theorem \ref{main-thm2}]
Suppose $F\in \mathscr{J}$ is an inner function of finite entropy with $\inn F' = B S_\mu$. 
By Theorem \ref{main-thm-inner}, we may write
$$
\mu=\sum_{i=1}^{\infty}\mu_i,
$$
where each measure $\mu_i$ is supported on a Beurling-Carleson set. Fix an integer $i \ge 1$.
We want to show that for a.e.~$x$ with respect to $\mu_i$, $F$ has a thick limit at $x$. By the elementary bound
$$
h_\varepsilon \, = \,  \frac{\varepsilon}{P_{\mu_\varepsilon}(z_{3\varepsilon, \varepsilon})} \, \lesssim \, \frac{\varepsilon^2}{\mu(I(x,\varepsilon))}
$$ 
and Theorems \ref{local-behaviour} and \ref{H-thickness-condition}, for any $C > 0$, 
$$
\bigl \{ z_{3\varepsilon, h}\, : \, 0 < \varepsilon \le \varepsilon_0, \,
C h_\varepsilon \le h \le \varepsilon \bigr \}
$$
describes the right part of a thick approach region, minus a Stolz angle, at $\mu_i$ a.e.~$x$. By the remark at the end of Section \ref{sec:abundance-radial}, we may assume that
$F$ has a non-tangential limit at $x \in \partial \mathbb{D}$ which lies in the open unit disk.
 In order to estimate $|F(z_{3\varepsilon, h}) - F(x)|$ from above, we connect $z_{3\varepsilon, h}$ and $x$ by a union of two line segments:
 $[z_{3\varepsilon, h}, z_{3\varepsilon,\varepsilon}] \cup [z_{3\varepsilon,\varepsilon}, x]$ and give an upper bound for the hyperbolic length of
 $F_\mu \bigl ([z_{3\varepsilon, h}, z_{3\varepsilon,\varepsilon}] \cup [z_{3\varepsilon,\varepsilon}, x] \bigr )$.

 Fix a positive real number $\eta >0$. We claim that for $\mu_i$ a.e.~$x$, we can choose the parameter $C = C(x, \eta)$ so that for any
 sufficiently small $\varepsilon > 0$, the hyperbolic length of $$F_{\mu_{i,\varepsilon}} \bigl ( [z_{3\varepsilon,C h_\varepsilon}, z_{3\varepsilon,\varepsilon}] \cup [z_{3\varepsilon,\varepsilon}, x] \bigr )$$ is less than $\eta$.
From Lemma \ref{fundamental-lemma}, we know that $\lambda_{F} \le \lambda_{F_{\mu_{i, \varepsilon}}}$ for any $i \ge 1$ and $\varepsilon > 0$.
 Thus a fortiori, the hyperbolic length of
$F \bigl ( [z_{3\varepsilon, h}, z_{3\varepsilon,\varepsilon}] \cup [z_{3\varepsilon,\varepsilon}, x] \bigr )$ is also less than $\eta$.
Therefore, once we prove the claim, we obtain the abundance of right parts of thick approach regions. 
An analogous argument will then show the abundance of left parts of thick approach regions, thereby proving the theorem.

It remains to prove the claim. The hyperbolic length of $F_{\mu_{i,\varepsilon}}([z_{3\varepsilon,\varepsilon}, x])$ can be estimated analogously to the radial approach, and so is less than $\eta/2$ when $\varepsilon > 0$ is small.
Since $\mu_{i, \varepsilon} \ge \mu_i^L = \mu_i|_{I^L}$ for any $\varepsilon > 0$, the radial limit $F_{\mu_{i,\varepsilon}}(x)$ lies in a compact subset of the unit disk.
A brief inspection of (\ref{eq:contraction at thick point}) shows that by selecting the constant $C = C(x, \eta)$ to be large, we can make the Euclidean length of $F_{\mu_{i, \varepsilon}}([z_{3\varepsilon,C h_\varepsilon}, z_{3\varepsilon,\varepsilon}])$ to be as small as we wish.
Consequently, we can choose $C > 0$ sufficiently large to ensure that the hyperbolic length of $F_{\mu_{i,\varepsilon}}([z_{3\varepsilon,C h_\varepsilon}, z_{3\varepsilon, \varepsilon}])$ is also less than $\eta/2$. The proof is complete.
\end{proof}

\section{Continuity of critical values}
\label{sec:finite}

In this section, we show that the singular value measures
$$
\nu_F = \sum_{c \in \crit F} (1-|c|) \cdot  \delta_{F(c)} + \underbrace{F_* (\sigma(F'))}_{\overline{\nu}_F}$$
vary continuously in $F \in \mathscr J$, thereby completing the proof of Theorem \ref{main-thm}.
More precisely, we show that if a sequence of inner functions $F_n \in \mathscr J$ converges stably to $F \in \mathscr J$, then
the measures $\nu_{F_n} \to \nu_F$ converge weakly.
It is enough to consider the case when the $F_n$ are finite Blaschke products as the general case follows after a diagonal argument.

Let $V \subset \mathbb{D}$ be a round disk compactly contained in $\mathbb{D}$ and $\mathcal U$ be the set of connected components of $F^{-1}(V)$. For each  $U_k \in \mathcal U$, choose a basepoint $p_k \in U_k$. Let $\varphi_k: \mathbb{D} \to U_k$ be the conformal map with $\varphi_k(p_k) = p_k$ and $\varphi_k'(p_k) > 0$.

 After passing to a subsequence, for each $U_k \in \mathcal U$, we can find connected components $U_{n,k}$ of $F_n^{-1}(V)$ such that $(U_{n,k}, p_k) \to (U_k, p_k)$ in Carath\'eodory sense. If we normalize the conformal maps $\varphi_{n,U_{n,k}}: \mathbb{D} \to U_{n,k}$ so that $\varphi_{n,U_{n,k}}(p_k) = p_k$ and $\varphi_{n,U_{n,k}}'(p_k) > 0$, then $F_{n, U_{n,k}} \to F_{U_k}$ converge uniformly on compact subsets of $\mathbb{D}$. 
Due to the discrepancy between  Carath\'eodory  and Hausdorff limits, the same connected component $U_{n,k}$ may appear in more than one sequence.

\begin{proof}[Proof of Theorem Theorem \ref{main-thm}: Continuity]
As a point on the unit circle can be thick for at most one domain $U_k \in \mathcal U$, the measures
 $\sigma(F')|_{(\partial U_k \cap \partial \mathbb{D})_{\thick}}$ are supported on disjoint sets. Given an $\varepsilon > 0$, choose finitely many disjoint closed subsets
$$E_k \subset (\partial U_k \cap \partial \mathbb{D})_{\thick}, \qquad k = 1, 2, \dots, N,$$
such that 
$
\bigl \|  \sum_{k=1}^\infty \sigma(F')|_{(\partial U_k \cap \partial \mathbb{D})_{\thick}} - 
 \sum_{k=1}^N \sigma(F')|_{E_k} \bigr \|_{\mathcal M} < \varepsilon.
$ 
 
  Since $F_{n, U_{n,k}} \to F_{U_k}$ converges uniformly on compact sets, we have
$$
\liminf_{n \to \infty} \sum_{\substack{\hat c \in \crit F_{n, U_{n,k}} \\ \escaping}} G_{\mathbb{D}}( 0, \hat c) \cdot \delta_{\hat c} \, \ge \, \sigma(F'_{U_k}),
$$
in the sense of measures. An application of Lemma \ref{green-quotient3} gives
$$
\liminf_{n \to \infty} \sum_{\substack{\hat c \in \crit F_{n, U_{n,k}} \\ \escaping}} G_{\mathbb{D}}( p_k, \hat c) \cdot \delta_{\hat c} \, \ge \,  P_{p_k}(\zeta) \cdot d\sigma(F'_{U_k}).
$$
Using the conformal invariance of the Green's function and Theorem \ref{radon-nikodym}, we get
$$
\liminf_{n \to \infty} \sum_{\substack{c \in \crit F_n \cap U_{n,k} \\ \escapingto E_k}} G_{U_{n,k}}(p_k, c) \cdot \delta_c
\, \ge \,
 |(\varphi_k^{-1})'(\zeta)|  \cdot P_{p_k}(\zeta) \cdot d\sigma(F')|_{E_k}.
 $$
Applying Lemmas \ref{green-quotient2} and \ref{green-quotient3} respectively shows
$$
\liminf_{n \to \infty} \sum_{\substack{c \in \crit F_n \cap U_{n,k} \\ \escapingto E_k}} G_{\mathbb{D}}(p_k, c) \cdot \delta_c
\, \ge \,
P_{p_k}(\zeta) \cdot d\sigma(F')|_{E_k}
 $$
and
\begin{equation}
\label{eq:finite2}
\liminf_{n \to \infty} \sum_{\substack{c \in \crit F_n \cap U_{n,k} \\ \escapingto E_k}} G_{\mathbb{D}}(0, c) \cdot \delta_c
\, \ge \,
 \sigma(F')|_{E_k}.
\end{equation}
In the above computation, we have used that $ |(\varphi_k^{-1})'(\zeta)| > 0$ on $E_k$ to avoid dividing by 0.
Summing (\ref{eq:finite2}) over the connected components of $F_n^{-1}(V)$, we obtain
\begin{align*}
\liminf_{n \to \infty} \sum_{\substack{c \,\in\, \crit F_n \,\cap\, F_n^{-1}(V) \\ \escaping}} G_{\mathbb{D}}(0, c) & \ge \sum_{k=1}^N \sigma(F')(E_k) \\
& \ge  \sum_{k=1}^\infty \sigma(F') \bigl ( (\partial U_k \cap \partial \mathbb{D})_{\thick} \bigr ) - \varepsilon \\
  & \ge
  \overline{\nu}_F(V) - \varepsilon,
\end{align*}
where in the last step, we have crucially used Theorem \ref{main-thm2}.
 As $\varepsilon > 0$ was arbitrary, we see that
\begin{equation}
\label{eq:finite1}
\liminf_{n \to \infty} \sum_{\substack{c \,\in\, \crit F_n \,\cap\, F_n^{-1}(V) \\ \escaping}} G_{\mathbb{D}}(0, c)  \ge \overline{\nu}_F(V).
\end{equation}
Since (\ref{eq:finite1}) holds for any round disk $V$ compactly contained in $\mathbb{D}$, any subsequential limit of the measures $\nu_{F_n}$ is at least $\nu_F$.
However, as the sequence $F_n \to F$ is stable, $\nu_{F_n} (\overline{\mathbb{D}}) \to \nu_{F} (\overline{\mathbb{D}})$, which implies that $\nu_{F_n} \to \nu_F$ weakly. The proof is complete.
\end{proof}

\appendix

\section{Compactness}
\label{sec:compactness}

It is well known that the closure of finite Blaschke products in the topology of uniform convergence of compact subsets is the unit ball in $H^\infty$ together with the modular constants, e.g.~see \cite{mashreghi-ransford}. If we only use finite Blaschke products vanish at the origin, then the closure consists of functions in the $H^\infty$ unit ball that vanish at the origin.

The following lemma says that if the entropy is bounded above, then the limit is necessarily an inner function:

\begin{lemma}
\label{J-compactness}
Suppose that $F_n \in \mathscr J$ is a sequence of inner functions of finite entropy with $F_n(0) = 0$, converging uniformly on compact subsets of the unit disk to
  $F$. If
\begin{equation}
\label{eq:J-compactness}
\sup \int_{|z|=1} \log |F'_n(z)| dm < \infty,
\end{equation}
then the limit function $F \in \mathscr J$.
\end{lemma}

\begin{proof}
According to the fundamental lemma \cite[Lemma 2.3]{inner},
$$
\lambda_{F_n} \ge | \inn F'_n | \lambda_{\mathbb{D}},
$$
where $\lambda_{F_n} = \frac{2|F'_n|}{1-|F_n|^2}$ and $\lambda_{\mathbb{D}} = \frac{2}{1-|z|^2}$.
Taking $n \to \infty$ and passing to a subsequence if necessary gives
$$
\lambda_{F} \ge | H | \lambda_{\mathbb{D}},
$$
where $H$ is a bounded holomorphic function. In view of Jensen's formula, the assumption (\ref{eq:J-compactness})
 rules out the possibility that $H$ is identically zero.
The lemma now follows from \cite[Lemma 3.6]{inner}. 
\end{proof}

\section{Basic properties of components}

\label{sec:preimage}

Let $F:\mathbb{D}\to\mathbb{D}$ be an inner function, $V\subset\mathbb{D}$ be a Jordan domain compactly contained in the disk and $U\subset F^{-1}(V)$ be a connected component of the pre-image. In this appendix, we show Lemmas \ref{component-is-jordan} and \ref{component-is-inner} from the introduction. We begin by studying the topological properties of $U$.

\begin{lemma}
\label{Same boundary}
The domains $U$ and $\mathbb{C} \setminus \overline{U}$ have the same boundary.
\end{lemma} 

\begin{proof}
The inclusion $\partial (\mathbb{C} \setminus \overline{U})\subseteq \partial U$ is automatic as
$$
\partial (\mathbb{C} \setminus \overline{U}) \, = \, \overline{(\mathbb{C} \setminus \overline{U})} \setminus (\mathbb{C} \setminus \overline{U}) \, = \, \overline{(\mathbb{C} \setminus \overline{U})} \, \cap \, \overline{U} \, \subset \, \overline{U}.
$$
We therefore need to show the other inclusion $\partial U \subseteq \partial(\mathbb{C} \setminus \overline{U})$. It is clear that if $z\in \partial U \cap \partial \mathbb{D}$, then $z \in \partial (\mathbb{C} \setminus \overline{U})$. Suppose instead that $z\in \partial U \cap \mathbb{D}$. Since $F$ is an open mapping, $F(z)\in \partial V$ and the image of any ball $B(z, \varepsilon)$ intersects $\mathbb{D} \setminus \overline{V}$, which also implies that $z \in \partial (\mathbb{C} \setminus \overline{U})$.
\end{proof}

\begin{lemma}
\label{Hinfinity}
Let $f: \mathbb{D} \to \mathbb{C}$ be a bounded analytic function and  $K$ be a compact subset of the plane with connected complement. Suppose that outside a measure zero subset $E \subset \partial \mathbb{D}$, the radial boundary values
$$
f(\zeta) := \lim_{r \to 1} f(r\zeta)
$$
are contained in $K$. Then, $f(\mathbb{D}) \subset \interior K$.
\end{lemma}

\begin{proof}
We assume that $f$ is non-constant, otherwise, the lemma is obvious.
Let $b \in \mathbb{C} \setminus K$. By Runge's theorem, there is a polynomial $p$ with $p(K) \subset \overline{\mathbb{D}}$ and $|p(b)| > 1$. Since $p \circ f$ is a bounded analytic function on the unit disk whose radial boundary values are at most 1 a.e., $\| p \circ f \|_{H^\infty} \le 1$. This prevents $b$ from being in the image of $f$. In other words, $f(\mathbb{D}) \subset K$.
As $f$ is an open mapping, $f(\mathbb{D}) \subset \interior K$.
\end{proof}

\begin{corollary}
\label{Holes Lemma}
Suppose $W \subset \mathbb{D}$ is a simply connected domain whose boundary intersects $\partial \mathbb{D}$ in a set of measure 0.
If  $F(\partial W\cap \mathbb{D}) \subset \partial V$, then $F(W)\subset V$.
\end{corollary}

\begin{proof}
Let $R: \mathbb{D} \to W$ be a conformal map. Since $W \subset \mathbb{D}$, by Loewner's lemma, $R^{-1}(\partial W \cap \partial \mathbb{D}) \subset \partial \mathbb{D}$ has measure zero. (One can see this by using the interpretation of harmonic measure as the hitting distribution of Brownian motion.) 
The corollary follows after applying Lemma \ref{Hinfinity} with $f = F \circ R$.
\end{proof}

We now recall some well-known facts from topology:
\begin{enumerate}
\item If $E \subset \hat{\mathbb{C}}$ is a closed connected subset of the sphere, then every connected component of $\hat{\mathbb{C}} \setminus E$ is simply-connected. See \cite[p.~139]{ahlfors}.

\item Suppose now that $E \subset \mathbb{C}$ is a compact set with connected interior. A bounded connected component of $\mathbb{C} \setminus E$  is called a {\em hole} of $E$. If $E$ has no holes, $\interior E$ is simply-connected and  $\partial (\interior E)$ is connected.
\end{enumerate}

\begin{lemma}
\label{simply connected}
The domains $U$ and $\hat{\mathbb{C}} \setminus \overline{U}$ are simply-connected, while $\partial U$ is connected.
\end{lemma}

\begin{proof}
By Corollary \ref{Holes Lemma}, $\overline{U}$ has no holes, so that $\hat{\mathbb{C}} \setminus \overline{U}$ consists only of the unbounded connected component, which is simply connected. Furthermore, by Lemma \ref{Same boundary}, $U = \interior \overline{U}$. By the second fact above, $U$ is simply connected and $\partial U$ is connected.
\end{proof}

\begin{lemma}
\label{locally connected}
The boundary of $U$ is locally connected.
\end{lemma}

Before giving the proof, recall that a topological space $X$ is {\em locally connected} if every point $x \in X$ admits a neighbourhood basis consisting of open, connected sets. A topological space $X$ is {\em weakly locally connected}  if for every point $x \in X$ and open set $U$ containing $x$, one can find a connected set $A \subset U$ which contains a neighbourhood $V$ of $\zeta$. As weak local connectivity is equivalent to local connectivity by \cite[p.~162]{munkres}, it is enough to show that $\partial U$ is weakly locally connected.

\begin{proof}[Proof of Lemma \ref{locally connected}]
It is not difficult to check the weak local connectivity of $\partial U$ near a point $z \in \partial U\cap \mathbb{D}$. Indeed, if $z \in \partial U\cap \mathbb{D}$ is not a critical point, then $F$ is a homeomorphism in a neighbourhood of $z$. On the other hand, if $z \in \partial U\cap \mathbb{D}$ is a critical point of order $m \ge 1$, then by Lemma \ref{Holes Lemma}, only one of the $m$ wedges of angle $\frac{2\pi}{m}$ around $z$ may be present in $U$, on which $F$ acts homeomorphically.

It remains to verify the weak local connectivity of $\partial U$ at a point $\zeta\in\partial U\cap\partial\mathbb{D}$.  Pick $0 < r_0 = r_0(\zeta) < 1/3$ so that $\partial U$ is not entirely contained in $B(\zeta, r_0)$. Consider the circular arcs $C_r = \partial B(\zeta,r) \cap \mathbb{D}$ centered at $\zeta$, with $0 < r < r_0$. As $C_r$ approaches the unit circle non-tangentially and $F$ is an inner function, for a.e.~$0 < r < r_0$, the intersection of $C_r$ and $\partial U$ is compactly contained in the unit disk.

Consider two such radii $r_1,r_2$ for which the intersections of $C_r$ and $\partial U$ do not go off the unit circle. By the local connectivity of $\partial U$ inside the unit disk, $\partial U$ can cross finitely many times from $C_{r_1}$ to $C_{r_2}$. By a {\em crossing}\/, we mean a connected component of $\partial U$ which lies between $C_{r_1}$ and $C_{r_2}$ and intersects both arcs. Indeed, if there were infinitely such crossings, then a Hausdorff limit would contradict the local connectivity of $\partial U$ inside the unit disk.

Given an $0 < r < r_0$, take two good radii $0 < r_1< r_2 < r$ as above. 
Since $\partial U$ is connected, every connected component of $\partial U \cap B(\zeta, r_2)$ must contain a crossing from $C_{r_1}$ to $C_{r_2}$. We may therefore enumerate the connected components of $\partial U \cap B(\zeta, r_2)$ as $A_1, A_2, \dots, A_n$, with $A_1$ being the connected component which contains which contains $\zeta$.
 Since $A_2, A_3, \dots, A_n$ are closed sets and don't contain $\zeta$, they are located at definite distance from $\zeta$.
Hence, $A_1$ contains an open neighborhood of $\zeta$ in $\partial U$. In other words, $\partial U$ is weakly
 locally connected at $\zeta$.
\end{proof}

To check that $\partial U$ is a Jordan domain, we use the following criterion from \cite[Theorem 1]{Fary} which may be viewed as a
converse to the Jordan curve theorem:

\begin{lemma}
\label{Converse to Jordan curve theorem}
Let $C$ be a compact set in the plane. If the complement of $C$ has two components, and every point of $C$ is an accessible boundary point of both components, then $C$ is a Jordan curve.
\end{lemma}

In the lemma above, a boundary point $p \in \partial \Omega$ is said to be {\em accessible} from $\Omega$ if $\Omega \cup \{ p \}$ is path-connected.

\begin{proof}[Proof of Lemma \ref{component-is-jordan}]
In Lemmas \ref{Same boundary}, \ref{simply connected} and \ref{locally connected}, we have seen that $U$ and $\hat{\mathbb{C}} \setminus \overline{U}$ are simply connected domains with $\partial U =  \partial(\hat{\mathbb{C}}\setminus\overline{U})$ locally connected.
By a theorem of Carathéodory, every point of $\partial U = \partial(\hat{\mathbb{C}} \setminus \overline{U})$ is accessible from both $U$ and $\hat{\mathbb{C}} \setminus \overline{U}$.
Since $\mathbb{C} = (\mathbb{C}\setminus\overline{U}) \cup \partial U \cup U$, Lemma \ref{Converse to Jordan curve theorem} tells us that $\partial U$ is a Jordan curve.
 \end{proof}

Having examined the topological properties of $U$, we turn to studying the component function $F_U$.

\begin{proof}[Proof of Lemma  \ref{component-is-inner}] 
We claim that for a.e.~$\zeta \in \partial \mathbb{D}$, the radial boundary value $\varphi(\zeta) := \lim_{r \to 1} \varphi(r\zeta)$ exists and lies inside the unit disk. Once we prove the claim, we are done since at such a point $\zeta$,
$$
\lim_{r \to 1} (\psi^{-1} \circ F \circ \varphi) (r\zeta)
$$
exists and has absolute value 1.

Let $E \subset \partial \mathbb{D}$ be the set of points $\zeta$ on the unit circle at which $\varphi$ and $F \circ \varphi$ have radial limits
with  $\varphi(\zeta)  \in \partial \mathbb{D}$. The above claim reduces to showing $E$ has measure zero.

If $\zeta \in E$, then $F$ has a limit along the path $\varphi([0,\zeta))$. In other words, $F$ has an asymptotic value at $\varphi(\zeta)$. As $F$ is a bounded analytic function, this asymptotic value (which lies in $\overline{V}$) must be the same as its radial boundary value \cite[Theorem 2.2]{cluster-sets}. 
Since $V$ is compactly contained in the unit disk and $F$ is an inner function, $\varphi(\zeta)$ is confined to a measure zero subset of $\partial \mathbb{D}$. In other words, $\varphi(E)$ has measure zero. Loewner's lemma implies that $E$ itself has measure zero.
\end{proof}

\section*{Acknowledgements}
The authors wish to thank the anonymous referee for numerous corrections and remarks that helped improve the exposition of this paper.
This research was supported by the Israeli Science Foundation (grant no.~3134/21).

\bibliographystyle{amsplain}

\end{document}